\pdfoutput=1
\documentclass[12pt]{article}
\usepackage{mathtools,amssymb,amsfonts,graphicx}
\usepackage{amsthm}
\usepackage{hyperref}
\usepackage{verbatim}
\usepackage{multicol,color,xcolor,longtable}
\usepackage{stmaryrd}
\usepackage{dsfont}
\usepackage{braket}
\usepackage{appendix}
\usepackage{enumitem}   
\usepackage[backend=biber,style=numeric,autocite=plain,sorting=none,maxbibnames=99,eprint=true,isbn=false]{biblatex}
\usepackage{url}
\definecolor{darkred}{rgb}{0.65,0.15,0}
\definecolor{darkgreen}{rgb}{.05,.5,.05}
\hypersetup{pdfborder={0 0 0},colorlinks=true,urlcolor=blue,citecolor=blue,linkcolor=darkred,linktocpage=true}

\newcommand{\nn}{\nonumber}

\newcommand{\alphaSet}{\{\alpha\}}
\newcommand{\HSet}{\{\Hc\}}
\newcommand{\QSet}{\{\Qp\}}

\newcommand{\nats}{\mathds{N}}
\newcommand{\ints}{\mathds{Z}}
\newcommand{\reals}{\mathds{R}}
\newcommand{\comp}{\mathds{C}}

\newcommand{\gaffine}{\gfrak}
\newcommand{\gfinite}{\mathring{\gfrak}}
\newcommand{\gloop}{\mathcal{L}({\gfinite})}
\newcommand{\hfinite}{\mathring{\mathfrak{h}}}
\newcommand{\cheisenberg}{\left\langle\Hc\right\rangle}
\newcommand{\pheisenberg}{\left\langle \Qp \right\rangle}
\newcommand{\kfinite}{\mathring{\kfrak}}
\newcommand{\pfinite}{\mathring{\mathfrak{p}}}
\newcommand{\taufinite}{\mathring{\tau}}
\newcommand{\ideal}{\mathcal{I}} 
\newcommand{\kfrakpar}{\mathfrak{P}} 

\newcommand{\basic}{\mathrm{R}(\Lambda_0)}
\newcommand{\adj}{\mathrm{ad}}
\newcommand{\sorep}[2]{\mathbf{#1}_{\so(#2)}}

\newcommand{\enn}{\mathfrak{e}_{9(9)}}
\newcommand{\eee}{\mathfrak{e}_{8(8)}}
\newcommand{\so}{\mathfrak{so}}
\newcommand{\kenn}{k(\mathfrak{e}_{9(9)})}

\newcommand{\sign}{\mathrm{sign}}
\newcommand{\modulo}[2]{#1\, \textrm{mod}\, #2}
\newcommand{\gsum}[2]{\underset{#1}{\overset{#2}{\overline{\sum}}}} 
\newcommand{\aeq}{\, =& \; }

\newcommand{\eq}{\, =\, }
\newcommand{\acoloneqq}{\, \coloneqq& \, }
\newcommand{\ecoloneqq}{\, \coloneqq \, }

\newcommand{\forwhich}{\; |\,}

\newcommand{\half}{\tfrac{1}{2}}

\newcommand{\Fc}{\mathcal{F}} 
\newcommand{\Uc}{\mathcal{U}} 
\newcommand{\Kc}{\mathcal{K}} 
\newcommand{\Ec}{\mathcal{E}} 
\newcommand{\Hc}{\mathcal{H}} 
\newcommand{\Pc}{\mathcal{P}}

\newcommand{\hfrak}{\mathfrak{h}}

\newcommand{\pfrak}{\mathfrak{p}}
\newcommand{\gfrak}{\mathfrak{g}}
\newcommand{\kfrak}{\mathfrak{k}}
\newcommand{\Nfrak}{\mathfrak{N}} 

\newcommand{\Ar}{\mathring{A}} 

\newcommand{\drm}{\mathrm{d}}

\newcommand{\be}{{\mathbf{e}}}

\newcommand{\Sh}{S^{\Hc}}
\newcommand{\Qr}{Q} 
\newcommand{\Qp}{R} 
\newcommand{\Qt}{\tilde{\Qp}} 
\newcommand{\verQ}{\mathrm{Ver}_{\Qp}} 
\newcommand{\verQfull}{\widetilde{\mathrm{Ver}}_{\Qp}} 
\newcommand{\Gmap}{G}

\newtheorem{theorem}{Theorem}

\newtheorem{proposition}[theorem]{Proposition}
\newtheorem{lemma}[theorem]{Lemma}
\newtheorem{corollary}[theorem]{Corollary}

\theoremstyle{remark}
\newtheorem{rem}[theorem]{Remark}

\theoremstyle{definition}

\makeatletter

\@addtoreset{equation}{section}
\makeatother

\begin{document}
\mbox{ }

\vspace{10mm}

\begin{center}
  {\LARGE \sc 
$\kfrak$-structure of basic representation\\[3mm] of affine algebras
   }
    \\[15mm]

{\large
Benedikt K\"onig
}

\vspace{8mm}
{\textit{Max-Planck-Institut f\"{u}r Gravitationsphysik (Albert-Einstein-Institut)\\
Am M\"{u}hlenberg 1, DE-14476 Potsdam, Germany}\\
\texttt{benedikt.koenig@aei.mpg.de}}


\end{center}

\vspace{35mm}

\begin{center} 
\hrule

\vspace{6mm}

\begin{tabular}{p{14cm}}
{\small%
This article presents a new relation between the basic representation of split real simply-laced affine Kac-Moody algebras and finite dimensional representations of its maximal compact subalgebra $\kfrak$. 
We provide infinitely many $\kfrak$-subrepresentations of the basic representation and we prove that these are all the finite dimensional $\kfrak$-subrepresentations of the basic representation such that the quotient of the basic representation by the subrepresentation is a finite dimensional representation of a certain parabolic algebra and of the maximal compact subalgebra. By this result we provide an infinite composition series with a cosocle filtration of the basic representation.
Finally, we present examples of the results and
applications to supergravity.}
\end{tabular}
\vspace{5mm}
\hrule
\end{center}
\thispagestyle{empty}

\newpage
\setcounter{page}{1}


\setcounter{tocdepth}{2}
\tableofcontents

\noindent\rule{\textwidth}{1.5pt}

\section{Introduction}
Affine Kac-Moody algebras, which play an important role in mathematics and physics, are by definition complex, but allow for different real slices from real forms \cite{rousseau1994forms}. One of these real forms is the split real form in which the Chevalley-Serre generators are real, for example the resulting algebra is the real span of the Chevalley-Serre generators and it is called split real Kac-Moody algebra.
The affine (split real) Kac-Moody algebras possess standard highest and lowest-weight representations which are fairly well studied \cite{Kac:1990gs}. Especially, for the basic representation $\basic$ with fundamental weight $\Lambda_0$ associated to the affine node $0$, there exists a realization in terms of vertex operators \cite{Kac:1990gs,Frenkel:1980rn,Segal:1981ap,Goddard:1988}. 

Additionally, every split real Kac-Moody algebra is equipped with a Cartan-Chevalley involution that defines the maximal compact subalgebra of the Kac-Moody algebra as its fixed point subalgebra. 
For finite split real Kac-Moody algebras 
$\gfinite$, the maximal compact subalgebra is a not necessarily indecomposable Kac-Moody algebra $\kfinite$ and the associated Lie group is compact. These algebras are very well understood. However, the scenario is completely different for affine split real Kac-Moody algebras $\gaffine$, for which the infinite dimensional maximal compact subalgebra $\kfrak$ is \textit{not} of Kac-Moody type, but admits infinitely many non-trivial ideals and does not have a Borel decomposition, but rather a filtered structure. 
Nevertheless, in \cite{Berman:1989} defining relations of generators for $\kfrak$ are developed and in \cite{Nicolai:2004nv} certain unfaithful finite dimensional representations, motivated by physics, are constructed, which is quite surprising since infinite dimensional Kac-Moody algebras do not have non-trivial finite dimensional representations. 

This result is generalized to a comprehensive understanding of finite dimensional representations of $\kfrak$ in terms of a (double) parabolic algebra $\kfrakpar$ and a Lie algebra homomorphism $\rho$ mapping $\kfrak$ to $\kfrakpar$ in \cite{Kleinschmidt:2021agj} and Appendix \ref{app:Reps}. The (double) parabolic algebra $\kfrakpar$ has a parabolic grading and therefore it contains infinitely many ideals which allow for an algorithmic construction of finite dimensional representations. The pull back of $\rho$ then defines ideals and representations of $\kfrak$, which preserve the parabolic grading. Thus, the finite dimensional representations of the maximal compact subalgebra $\kfrak$ are quite different from the infinite dimensional representations of the affine split real Kac-Moody algebra $\gaffine$. 
Despite it, understanding the action of the maximal compact subalgebra on a representation of a split real Kac-Moody algebra reveals fascinating additional structure of the representations and has important applications.\footnote{
\textbf{Lie algebra structure and fineness.}
For a Lie algebra $\mathfrak{l}$ with a representation $V$, a \textit{$\mathfrak{l}$-structure} of $V$ is a family of $\mathfrak{l}$-representations $V_i\subset V$. For two $\mathfrak{l}$-structures $A$ and $B$, we call $A$ \textit{finer} than $B$ if $B\subset A$. If for a $\mathfrak{l}$-structure $A$ holds $B\subset A$ for all $\mathfrak{l}$-structures $B$, we call $A$ the \textit{finest $\mathfrak{l}$-structure.}
}

For finite split real Kac-Moody algebras $\gfinite$, the complete information about the action of the maximal compact subalgebra $\kfinite$ on a $\gfinite$-representation is contained in the decomposition into irreducible $\kfinite$-representations, also known as branching. 
This is a standard technique of representation theory, because $\kfinite$ is a finite Kac-Moody algebra and its representations are completely reducible. However, this is not the case for affine split real Kac-Moody algebras. Actually, due to the intriguing structure of the maximal compact subalgebra $\kfrak$ and its representations, we do not even expect that a $\gaffine$-representation can be decomposed in a direct sum of simple or semisimple $\kfrak$-representations. Therefore, the best one may achieve is to find all $\kfrak$-subrepresentations of a $\gaffine$-representation and if possible to provide a (possibly infinite) $\kfrak$-composition series of the $\gaffine$-representation. But to our knowledge, very little is known about proper $\kfrak$-subrepresentation of highest-weight representation of $\gaffine$. 
Nonetheless, from physics applications (in particular supergravity in two dimensions), we expect that there exists proper non-trivial $\kfrak$-subrepresentations of the basic representation. In this context, this work is a first step towards a comprehensive understanding of the $\kfrak$-subrepresentations of highest weight representations of $\gaffine$ and towards understanding possible infinite $\kfrak$-composition series of the highest weight representations. 

Our main results is to reveal the complete/finest $\kfrak$-structure of the basic representation $\basic$ (level one representation, central charge one) of a split real untwisted simply-laced affine Kac-Moody algebra $\gaffine$, which is organized in three parts.
First, for each $N\in\nats_0$ we provide an inequivalent $\kfrak$-subrepresentation $W_N\subset\basic$ with finite co-dimension as the kernel of a surjective $\kfrak$-homomorphism $\Gmap_N$ from the basic representation to the finite dimensional $\kfrak$-representation $\verQ(N)$, such that $\kfrak$ commutes with $\Gmap_N$. Therefore, $\Gmap_N$ is a projection onto $\verQ(N)\simeq \basic/W_N$, where $\verQ(N)$ is truncated Verma module of a certain `parabolic Heisenberg algebra' with generators $\Qp$. We provide the representations $\verQ(N)$ and the projection $\Gmap_N$ explicitly. 

Second, we show that all $\kfrak$-subrepresentations of the basic representation, for which the quotient of the basic representation by the subrepresentation is equivalent to a finite dimensional $\kfrakpar$-representation are equivalent to $W_N$ or a subrepresentation of $W_N$ for some $N\in\nats_0$. 

Third, we provide the cosocle filtration (definitions are provided in Appendix \ref{app:Cosocle_Filtration}) of the basic representation under $\kfrak$ and we show that the cosocle filtration is an infinite composition series. It implies that as a vector space the basic representation is equivalent to a sum of infinitely many finite dimensional semisimple $\kfrak$-representations.
For the convenience of the reader, we provide these results in form of a theorem here, but for which it is necessary to refer to concepts in the main text.

\begin{theorem}
\label{the:CompleteDecomposition}
For $\gaffine$ a split real untwisted simply-laced affine Kac-Moody algebra with maximal compact subalgebra $\kfrak$ and basic representation $\basic$ the statements \textit{a}, \textit{b} and \textit{c} hold.
\begin{enumerate}[label=\alph*.)]
\item For all $N\in\nats_0$, there exists a finite dimensional $\kfrak$-representation $\verQ(N)$ given in Proposition \ref{prop:VerQRepresentation} and a surjective $\kfrak$-homomorphism $\Gmap_N:\, \basic \rightarrow \verQ(N)$ in \eqref{eq:GMap}, which commutes with $\kfrak$ and projects $\basic$ onto $\verQ(N)$. The kernel of $G_N$ is a proper non-trivial $\kfrak$-subrepresentation $W_N=\mathrm{Ker}(G_N)\subset\basic$.
\item If $\basic$ projects onto a finite dimensional $\kfrak$-representation which is also a $\kfrakpar$-representation by $\rho$ in \eqref{eq:LieHomDouble}, then there exists an $N\in\nats_0$ such that this representation is equivalent to $\verQ(N)$ or to a quotient of $\verQ(N)$ by a subrepresentation.
\item The family $(W_N)_{N\in\nats}$ of $\kfrak$-invariant subspaces in $\basic$ is the infinite composition series of $\basic$ with a cosocle filtration (definitions in Appendix \ref{app:Cosocle_Filtration}) .
\end{enumerate} 
\end{theorem}

These results find fascinating applications in physics.
(Gauged) supergravity theories in various dimensions are an important class of theories describing fundamental interaction and their mathematical structure is governed by Kac-Moody algebras \cite{Julia:1980gr}. While the bosonic fields of the theories take values in representations of the split real Kac-Moody algebra, the fermionic fields take values in representations of its maximal compact subalgebra. Therefore, supersymmetry, a symmetry between fermions and bosons, maps between representations of the split real Kac-Moody algebra and its compact subalgebra. 
This makes it necessary to understand how representations of the Kac-Moody-algebras project onto representations of the maximal compact subalgebra.
Supergravity theories in $D\geq 3$ have finite Kac-Moody symmetries and therefore the representation theory and the branching rules necessary to understand these theories are known. This makes it also possible to extend these supergravity theories to gauged supergravity theories. 
However, in $D=2$ dimensions, the supergravity theory has an (on-shell) affine Kac-Moody symmetry \cite{Geroch:1970nt,Breitenlohner:1986um,Julia:1980gr,Nicolai:1987kz,Nicolai:1988jb,Nicolai:1998gi}, under which the scalar fields transform in the adjoint representation and the gauge fields transform in the basic representation \cite{Samtleben:2007an}. While the adjoint representation is not a highest weight representation and the $\kfrak$-subrepresentation are understood, it is necessary to develop the $\kfrak$-subrepresentation of the basic representation to understand the supersymmetry of the gauge fields. 
Therefore, the full gauged supergravity theory in $D=2$ dimensions has not been constructed yet. Nevertheless a specific example is given in \cite{Ortiz:2012ib} and the purely bosonic sector is derived in \cite{Samtleben:2007an, Bossard:2018utw, Bossard:2021jix,Bossard:2022wvi, Bossard:2023wgg}. The result of this work is the next step to derive the full gauged supergravity theory in two dimensions. In a related context, the results find applications in teleparallel gravity in two dimensions \cite{Bossard:2024gry}.\\

This work is a starting point for further considerations in this direction in mathematics and physics. First, using the vertex operator construction for non-simply-laced affine Kac-Moody algebra \cite{Goddard:1986ts}, we expect our construction and results to extend to the basic representation of non-simply-laced affine Kac-Moody algebras. 
Second, in a more sophisticated development, our construction may be generalized to all highest-weight representations of affine Kac-Moody algebras. This may be achieved using the DDF construction \cite{DelGiudice:1971yjh} for further highest-weight modules of affine algebras, which is similar to the vertex operator realization.
Third, understanding the $\kfrak$-subrepresentations of highest weight representations of an affine Kac-Moody algebra paves the way to understand hyperbolic Kac-Moody algebras under the action of its maximal compact subalgebra. This is because a hyperbolic Kac-Moody algebra decomposes into sums of highest-weight representations of the associate affine Kac-Moody algebra \cite{Feingold_1983,Kac1988,Bauer_Violette_1997}.
However, this is clearly quite a challenging undertaking since there remains a lot to discover for hyperbolic Kac-Moody algebras \cite{Gebert:1994mv} but the result has important applications in physics.
For instance, it is expected that hyperbolic Kac-Moody algebras are the underlying mathematical structure of space-time \cite{Damour:2002et,Damour:2004zy,Damour:2005zs,Nicolai:2005su} governed by physical theories with hyperbolic symmetry \cite{Kleinschmidt:2005gz,Kleinschmidt:2006ad}. However, to fully describe these theories, the action of the maximal compact subalgebra on the hyperbolic Kac-Moody algebra must be understood. But even before this is fully achieved, our results have applications to these models, because (gauged) supergravity in two dimensions may serve as a toy model to understand the mechanism of how space time emerges from an indefinite Kac-Moody symmetry.

\paragraph{Outline. }
Section \ref{sec:AffineAlgebra_BasicRep} follows \cite{Kac:1990gs} closely to introduce the affine Kac-Moody algebras and the vertex operator realization for the basic representation. The maximal compact subalgebra of the Kac-Moody algebra is then defined in
Section \ref{sec:Maximal_Compact_Subalgebra}, which splits in two subsections. While Subsection \ref{subsec:Maximal_compact_subalgebra} provides the definition and notation of the maximal compact subalgebra, in Subsection \ref{subsec:Parabolic_algebra} we extend the construction of \cite{Kleinschmidt:2021agj} to a double parabolic algebra and apply it to prove an identity for ideals and a certain group action on the double parabolic algebra. Prepared with these identities, we use Section \ref{sec:ActionOnBasic} to analyze the action of the maximal compact subalgebra on the basic representation. This includes new notations and an important Proposition, which allows us to provide the finest $\kfrak$-structure of the basic representation and the infinite composition series in Section \ref{sec:CompleteDecomposition}. In Subsection \ref{subsec:Embedding_k_rep} we provide the infinite filtered set of $\kfrak$-subrepresentations of the basic representation and the projection of the basic representation onto finite dimensional $\kfrak$-representations. In Subsection \ref{subsec:Infinite_Limit} we show that the infinite set of $\kfrak$-subrepresentations is indeed the infinite $\kfrak$-composition series of the basic representation with cosocle filtration. 
In the final section \ref{sec:Applications_and_examples}, we prove an identity to efficiently evaluate the projection of the basic representation onto finite dimensional $\kfrak$-representations. This identity allows us to analyze explicit examples of how to project the basic representation onto $\kfrak$-representations and about the $\kfrak$-subrepresentations of the basic representation. Finally, we emphasis on $\gaffine=\enn$ and its application to supergravity in two dimensions. 
In Appendix \ref{app:Reps} we provide a construction of representations of the double parabolic algebra and Appendix \ref{app:PolySum} defines a generalizations of sums over polynomials. In Appendix \ref{app:Cosocle_Filtration} we provide frequently used definitions and notations as for example the definition of cosocle, cosocle filtration and infinite composition series. In Appendix \ref{app:ExpansionFirst order} we provide concrete expressions for abstract objects used in this work and, for the convenience of the reader, we give supplementary details of certain proofs in Appendix \ref{app:Appendix_Details}.

\paragraph{Acknowledgment.}
I would like to thank in particular Hermann Nicolai and Axel Kleinschmidt as well as Niklas Beisert, Martin Cederwall, Mattia Cesaro, Franz Ciceri, Alex Feingold, Matthias Gaberdiel, Jakob Palmkvist, Siddhartha Sahi and Henning Samtleben for useful discussions. 
Also, I like to thank Bernardo Araneda, Serena Giardino, Axel Kleinschmidt and Hermann Nicolai for helpful comments to and on the manuscript. Furthermore, I would like to thank the Konrad-Adenauer-Stiftung for founding and supporting me on this work and the IMPRS for Mathematical and Physical Aspects of Gravitation, Cosmology and Quantum Field Theory for the support. 

\section{Affine algebra and basic representation}
\label{sec:AffineAlgebra_BasicRep}
This section introduces the infinite dimensional untwisted affine Kac-Moody algebras and its simplest highest-weight representation, the basic representation. Details of definitions and proofs are in \cite{Kac:1990gs}.

Let $A = (a_{ij})_{0\leq i,j \leq r} $ be an untwisted symmetric affine generalized indecomposable Cartan matrix of rank $r$ over $\comp$ and let $\Ar = (a_{ij})_{1\leq i,j \leq r}$ be the finite Cartan matrix of the same type and rank $r$. 
Let $\gaffine$ be the split real form of the affine Kac-Moody algebra $\gaffine(A)$ and $\gfinite$ the split real form of the finite Kac-Moody algebra $\gaffine(\Ar)$.
Then $\gaffine$ is isomorphic to the loop algebra $\gloop$ of the Lie algebra $\gfinite$ with a central extension $K$ and a derivation $\drm$.
This allows us to parametrize the affine algebra $\gaffine$ in terms of the generators of $\gloop$, $K$ and $\drm$. Therefore, we first introduce the finite Lie algebra $\gfinite$.

\paragraph{Finite Lie algebra $\gfinite$.}
Let $(\mathring{\hfrak},\Pi,\Pi^{\vee})$ be are realization of $\Ar$. Then, $\Pi = \{\alpha^{i}\}_{1\leq i\leq r}$ is the set of simple roots of $\gfinite$, $\Pi^{\vee}$ is the set of dual simple roots, which is a basis of the Cartan subalgebra $\hfinite$ and the $\ints$-span of the simple roots is the root lattice $\Qr = \sum_{1\leq i \leq r} \alpha_i \ints$. The root lattice admits a natural bilinear form $(\alpha_i|\alpha_j) = \Ar_{ij}$. In terms of the bilinear form, the set of roots of $\gfinite$ is $\Delta = \{\alpha \in \Qr\forwhich (\alpha|\alpha) = 2\}$. From now on, we usually use the Greek letters $\alpha,\beta$ for roots and $\gamma, \delta$ for elements of the root lattice. 
On the root space exists a \textit{cocycle} \cite{Kac:1990gs}, also called \textit{asymmetric function}, $\epsilon \colon \Qr \times \Qr\rightarrow \Set{\pm1}$, which satisfies
\begin{align}
\label{eq:N_CocycleId}
\epsilon(\gamma + \gamma',\delta )\aeq\epsilon(\gamma,\delta )\,\epsilon(\gamma',\delta )\nn\\
\epsilon(\delta,\,\gamma+ \gamma')\aeq\epsilon(\delta ,\gamma )\,\epsilon(\delta,\gamma')\nn\\ 
\epsilon(\gamma,\gamma)\aeq (-1)^{(\gamma|\gamma)}\,.
\end{align}
The Lie bracket can be defined in terms of the cocycle. 
To carry this out, let $\gfinite_\alpha \subset \gfinite$ be the root space of $\alpha$ and let $E^\alpha \in \gfinite_{\alpha}$ be an element of this root space, where we define for convenience $E^\gamma=0$ for $\gamma\notin\Delta$. Let $H^\gamma\in \hfinite$ be the dual of $\gamma\in \Qr$ with respect to the bilinear form $(\cdot|\cdot)$, for example $\alpha(H^\gamma) = (\alpha|\gamma)$. Then, in terms of the two roots $\alpha,\beta\in\Delta$,
\begin{align}
\label{eq:Com_Finite_Algebra}
\left[H^\alpha,H^\beta\right] \aeq 0\nn\\
\left[H^\alpha,E^\beta\right] \aeq (\alpha|\beta)E^\beta\nn\\
\left[E^\alpha,E^\beta\right] \aeq -\delta_{\alpha,-\beta} H^\alpha\, + \, \epsilon(\alpha,\beta)E^{\alpha+\beta}
\end{align}
defines a Lie bracket on $\gfinite$.
The Lie algebra $\gfinite$ is equipped with the Killing form $\kappa$ which is symmetric, bilinear and invariant under the Lie bracket
\begin{align}
\kappa(H^\alpha,H^\beta) \eq (\alpha|\beta)\, \qquad 
\kappa(H^\alpha,E^\beta) \eq 0\, \qquad
\kappa(E^\alpha,E^{\beta}) \eq -\delta_{\alpha,-\beta}\, .
\end{align}
This sets up the split real finite Lie algebra $\gfinite$, from which it is convenient to define the affine Kac-Moody algebra $\gaffine$.


\paragraph{Affine algebra $\gaffine$.}
The split real affine Kac-Moody algebra $\gaffine$ can be parametrized in terms of the centrally extended loop algebra of $\gfinite$ with an additional derivation. Henceforth, we use this parametrization for the affine Kac-Moody algebra. Therefore, if $\reals[t,t^{-1}]$ are the Laurent polynomials in $t$, $K$ the central element and $\drm$ the derivation, then the loop parametrization of $\gaffine$ is
\begin{align}
\gaffine\eq \gfinite \otimes \reals\left[t,t^{-1}\right] \,\oplus\, \reals K\,\oplus\, \reals \drm\, ,
\end{align}
where the first summand is the loop algebra. For the loop generators we use the notation $x_n = x\otimes t^n$
with $x\in\gfinite$ and $n\in\ints$. The index $n$ of the loop generators is called the \textit{loop index}. 
In this parametrization, the Lie bracket of the affine Kac-Moody algebra is
\begin{align}
\label{eq:Lie_Bracket_Affine}
\left[x_m,y_n\right] \aeq \left[x,y\right]\otimes t^{m+n} + m\delta_{m,-n} \kappa(x,y) K
\nn\\
\left[\drm,x_m\right] \aeq -m x_m\nn\\
\left[K,x_m\right] \aeq \left[K,\drm\right] \eq 0 \, .
\end{align}
The affine algebra $\gaffine$ has two subalgebras which are important for the rest of the paper. The first subalgebra is spanned by the loop generators of index $0$. This subalgebra is isomorphic to the finite Lie algebra $\gfinite$, which motivates to use the same symbol for it.

The second important subalgebra is the Heisenberg subalgebra. The \textit{Heisenberg subalgebra} is the centrally extended loop algebra of the Cartan subalgebra\footnote{The Heisenberg subalgebra is \textit{not} the Cartan subalgebra of $\gaffine$.}
\begin{align}
\hfinite\otimes\reals\left[t,t^{-1}\right]\, \oplus\, \reals K = \left\langle {H_n,\,K\forwhich n\in\ints,H\in\hfinite}\right\rangle.
\end{align}

The affine algebra has no non-trivial finite dimensional representations. However, it has standard highest and lowest-weight representations. The `simplest' highest-weight representation is the basic representation, which is the subject of the next paragraph.

\paragraph{Basic representation.} The basic representation $\basic$ of $\gaffine$ is the infinite dimensional highest-weight representation with fundamental weight $\Lambda_0$, where $0$ is the index of the affine node. A realization of $\basic$ is given in terms of vertex operators which is essential for the rest of this work. Therefore, this paragraph sets up the vertex operator representation and notation, following \cite{Kac:1990gs}. 

In the \textit{vertex operator realization} of the basic representation, the representation space is
\begin{align}
\basic \eq  \reals[\Qr ] \otimes S\big(\bigoplus_{j<0}(\hfinite\otimes t^j)\big)\, ,
\end{align}
where $S$ is the symmetric algebra and $\reals[\Qr ]$ is the group algebra of the root lattice with the embedding $\gamma\mapsto e^\gamma$ for $\gamma\in \Qr$. 
Thus, the elements of the basic representation are parametrized in a tensor product of the root lattice $\Qr$ and polynomials in roots with an additional negative subscript
\begin{align}
\basic = \left\langle
e^{\gamma} \otimes f \forwhich \gamma \in \Qr\, , \; f \in \reals \left[\{\alpha\in\Delta\}_{-\nats}\right] \right\rangle\; .
\end{align}
Here and in the following, using the notation in \eqref{eq:SetNotation}, $f=f(\{\alpha\in\Delta\})$ is always a polynomial of the roots with an additional negative index. 

To introduce the action of $\gaffine$ in the vertex operator realization of the basic representation we introduce a representation of the indexed roots. 
For each $\alpha\in\Delta$ and $n \in \nats$, the negative indexed root $\alpha_{-n}$ acts by multiplication on $f$ and the positive indexed root $\alpha_{n}$ acts as a derivative on $f$ with
\begin{align}
\alpha_{n}\cdot \beta_{-m} \eq n\delta_{n,-m}(\alpha|\beta)\,.
\end{align}
Using this action, the \textit{vertex operator} acts on the states in the basic representation by
\begin{align}
\label{eq:N_Def_VOA_Gamma}
\Gamma^\alpha (z)\, \left(e^{\gamma} \otimes 
f\right) \, = \, z^{(\alpha,\gamma)}\, \epsilon(\alpha,\gamma) \; e^{\gamma + \alpha} \otimes \left(e^{\sum_{k\geq 1}\frac{z^k}{k}\alpha_{-k}}\,e^{-\sum_{k\geq 1}\frac{z^{-k}}{k}\alpha_{k}}\right)
f\; .
\end{align}
This vertex operator action, allows to define the basic representation $\gaffine\rightarrow \mathrm{End}(\basic)$. For all $\alpha\in\Delta$, $\gamma\in Q$ and $n\in\nats$ the basic representations is given by 
\begin{align}
\label{eq:N_VOARepresentation}
E^\alpha_n &\;\mapsto\; \Gamma^\alpha(z)\Big|_{z^{-n-1}},\nn\\ 
H^\alpha_n\left(e^{\gamma} \otimes f\right) &\;\mapsto\;
\begin{cases}
(\alpha,\gamma) \left(e^{\gamma} \otimes f\right) & \text{for $n=0$}\, \\
e^{\gamma} \otimes \alpha_n\, f & \text{for $n\neq 0$}\, 
\end{cases}\, ,\nn\\
K& \;\mapsto\; 1\nn\\
\drm  & \;\mapsto\; \sum_{i,j=1}^r\left(\half H_0^{\alpha^i}{\Ar}^{-1}_{ij}H_0^{\alpha^j} + \sum_{n\geq 1} H_{-n}^{\alpha^i}\Ar^{-1}_{ij}H_n^{\alpha^j}\right)
\end{align}
where $|_{z^{n}}$ is the projection of a polynomial in $z$, $z^{-1}$ on the coefficient of $z^n$ in the monomial basis $({z^k})_{k\in\ints}$.
The central element multiplies with the central charge one, also known as level one, and the eigenvalue of the derivation $\drm$ on a state is the loop level of this state. 
This fixes our notation of the basic representation. However, at the end of the section, we add some useful comments and frequently used notation on the vertex operator action.
The vertex operator action takes a particular simple form in terms of Schur polynomials on the states $e^\gamma \otimes 1\in\basic$. Therefore, let us introduce the notion of these states and the conventions of Schur polynomials in this work. 

\subparagraph{Schur polynomials and maximal states.}
For each $N\in\ints$ and the set notation $\alphaSet \coloneqq\{\alpha_{-k}\forwhich k\in\nats\}$, the \textit{Schur polynomial} $S_N(\alphaSet)$ is defined by the expansion
\footnote{In most literature, the Schur polynomials $S_\lambda(y_1,y_2, \ldots )$ are symmetric polynomials in the variables $y_i$ and are defined in terms of a partition $\lambda$ of $N$. This is related to our definition of Schur polynomials. Setting $\lambda = (N)$ and rewriting the Schur polynomial $S_\lambda(y_1,y_2, \ldots )$ in terms of symmetric power polynomials $\alpha_{-j} = \sum_i y_i^j$ gives the  Schur polynomials $S_N(\alphaSet)$ .} 
\begin{align}
\label{eq:Schurdef}
&\sum_{N\geq 0} z^N S_{N}(\alphaSet) \, \eq \, e^{\sum_{n\geq 1}\frac{z^n}{n}\alpha_{-n}}\, , \nn\\
&S_{n\leq-1}(\{\alpha\}) \eq 0\, ,\qquad
S_0(\{\alpha\}) \eq 1,\qquad  S_{1}(\{\alpha\}) \eq \alpha_{-1},\ldots \quad
\end{align}
The first few Schur polynomials are provided in \eqref{eq:Schur_x_Examples}.
Two important identities for Schur polynomials derive from \eqref{eq:Schurdef} by differentiation. Differentiating \eqref{eq:Schurdef} w.r.t. $\alpha_{-n}$ and w.r.t. $z$ gives for all $n\in\nats$, $N\in\nats_0$ 
\begin{align}
\label{Schur1}
\frac{\partial S_N(\alphaSet)}{\partial \alpha_{-n}}\,=\, \frac1{n} S_{N-n}(\alphaSet)\,,\qquad
N S_N(\alphaSet) \,=\, \sum_{n\geq 1} \alpha_{-n} S_{N-n}(\alphaSet)\,.
\end{align}
The Schur polynomials substantially simplify the action of the vertex operators on maximal states. Therefore, let us define these states.

A state of the basic representation is called \textit{maximal state} if it vanishes under the action of the positive indexed Heisenberg generators. This set of maximal states is
\begin{align}
\left\langle e^\gamma \otimes 1\in\basic \forwhich \gamma\in \Qr \right\rangle .
\end{align}
By the definition of maximal states, the right exponential in the vertex operator acts trivial and therefore, for $m\in\ints$, the action of $E^\alpha_m$ on a maximal state $e^\gamma\otimes 1$ takes a closed simple form in terms of Schur polynomials
\begin{align}
E^\alpha_m\, e^\gamma\otimes 1 \aeq \epsilon(\alpha,\gamma)\,e^{\gamma+\alpha} \otimes S_{-m-(\gamma,\alpha)-1}(\alphaSet)\, .
\end{align}
On these states it is convenient to evaluate the action of the Heisenberg algebra. For each $n\in\nats$, two often used relations for the action of the Heisenberg algebra are 
\begin{align}
\label{eq:SchurId}
H^\alpha_n \left( \,e^{\gamma}\otimes S_{N}(\{\alpha\}) \right) &=  2\, e^{\gamma}\otimes S_{N-n}(\{\alpha\})\, ,\nn\\
e^{\gamma}\otimes S_{N}(\{\alpha\}) &= \frac{1}{N} \sum_{n=1}^N H^\alpha_{-n}\,e^{\gamma}\otimes S_{N-n}(\{\alpha\})\, ,
\end{align}
which derive straightforward from \eqref{eq:N_VOARepresentation} and \eqref{Schur1}.

In this section we introduced the affine algebra and the basic representation in terms of the vertex operator realization. We present the maximal compact subalgebra of the affine algebra in the next section.

\section{Maximal compact subalgebra and representations}
\label{sec:Maximal_Compact_Subalgebra}
Here, we introduce the maximal compact subalgebra of an affine Kac-Moody algebra and the parabolic algebra \cite{Kleinschmidt:2021agj}. The maximal compact subalgebra of an affine algebra is not of Kac-Moody type, for example it admits infinitely many ideals. These ideals allow for finite dimensional representations of the maximal compact subalgebra, which where developed in \cite{Nicolai:2004nv,Kleinschmidt:2021agj, Kleinschmidt:2006dy}. 
This section provides the necessary identities and notation to analyze the action of the maximal compact subalgebra on the basic representations. 

In Subsection \ref{subsec:Maximal_compact_subalgebra} we define the maximal compact subalgebra as the fixed point set of the Cartan-Chevalley involution. In Subsection \ref{subsec:Parabolic_algebra} we extend the results of \cite{Kleinschmidt:2021agj} to a double parabolic algebra and a Lie algebra homomorphism from the maximal compact subalgebra to the double parabolic algebra. This Lie algebra homomorphism allows us to prove a generating identity for these ideals and derive the action of a certain group element on the parabolic algebra.

\subsection{Maximal compact subalgebra}
\label{subsec:Maximal_compact_subalgebra}
In this section we introduce the maximal compact subalgebra of an affine Kac-Moody algebra and the notation of the maximal compact subalgebra for this paper.

The \textit{maximal compact subalgebra} of a split real Kac-Moody algebra is the fix point subalgebra of the Cartan-Chevalley Lie algebra involution. For the finite split real Lie algebra $\gfinite$ and $\alpha\in\Delta$, the Cartan-Chevalley involution $\taufinite$ acts by 
\begin{align}
\label{eq:N_ChavellyInvolutionA1fin}
\taufinite(H^\alpha) \eq H^{-\alpha}\eq -H^{\alpha}\,, \qquad
\taufinite(E^\alpha) \eq E^{-\alpha}\,. 
\end{align}
The maximal compact subalgebra $\kfinite\subset\gfinite$ is the $+1$ eigenspace of $\taufinite$ and has strictly negative signature of the killing form. The $-1$ eigenspace is the non-compact orthogonal complement $\pfinite\subset\gfinite$ with strictly positive signature of the killing form.

Let us turn to the maximal compact subalgebra of the affine Kac-Moody algebra $\gaffine$. 
The Cartan-Chevalley involution $\tau$ of the affine algebra acts on the loop generators with the Cartan-Chevalley involution of the finite algebra and on the loop index it acts by inversion. On $\gaffine$, the involution is given by
\begin{alignat}{6}
\label{eq:N_ChavellyInvolutionA1aff}
\tau(H^\alpha_n) \eq -H^{\alpha}_{-n}\,, \qquad
\tau(E^\alpha_n) \eq E^{-\alpha}_{-n}\,, \qquad
\tau(K) \eq -K\,,\qquad  
\tau(\drm) \eq -\drm\, . 
\end{alignat}
The maximal compact subalgebra $\kfrak \in \gaffine$ is the $+1$ eigenspace subalgebra of $\tau$ and the $-1$ eigenspace is called the non-compact orthogonal complement $\pfrak\subset\gaffine$.
In particular, the central element $K$ and the derivation $\drm$ are not in the maximal compact subalgebra, but for every $\alpha\in\Delta$ and $n\in\ints$, the generators  
\begin{align}
\Ec^\alpha_n \aeq E_n^\alpha+E_{-n}^{-\alpha}\, ,\
\Hc_n^\alpha \eq \half\left(H_n^\alpha-H_{-n}^\alpha\right)\,\in\kfrak\, .
\end{align}
are in the maximal compact subalgebra.
Not all of these elements are linearly independent, because $\Hc_n^\alpha = -\Hc_{-n}^\alpha$ and $\Ec^\alpha_n = \Ec_{-n}^{-\alpha}$ and for a fixed $n$ there are $r$ independent $\Hc_n^\alpha$. However, it is useful to have introduced the notation of $\Hc_n^\alpha$ and $\Ec_n^\alpha$ for all combinations of roots and integers, which also span the maximal compact subalgebra
\begin{align}
\label{eq:Generatingkfrakpar}
\kfrak \aeq \left\langle \Ec_n^\alpha,\, \Hc_n^\alpha\forwhich \alpha\in\Delta,\, n\in\ints\right\rangle\, 
\end{align}
The Lie bracket is induced from the Lie bracket of $\gaffine$ and for $\alpha,\beta \in\Delta$, $n,m\in\ints$ it is
\begin{align}
\label{eq:Compact_Subalg_Commu}
\left[\Hc^\alpha_m,\Hc^\beta_n \right]\aeq 0\nn\\
\left[\Hc^\alpha_m,\Ec^\beta_n\right] \aeq \half (\alpha|\beta)\left(\Ec^\beta_{n+m} -\Ec^\beta_{n-m}\right)\nn\\
\left[\Ec^\alpha_m,\Ec^\beta_n\right] \aeq - 2\left(\delta_{\alpha,-\beta}+\delta_{\alpha,\beta}\right)\Hc^\alpha_{m+n} + \epsilon(\alpha,\beta)\left(\Ec^{\alpha+\beta}_{m+n}+\Ec^{\alpha-\beta}_{m-n}\right). 
\end{align}
Two comments on the structure of the Lie bracket are appropriate. First, the projection of the Heisenberg algebra to the maximal compact subalgebra is an abelian subalgebra, called the \textit{compact Heisenberg algebra}
\begin{align}
\label{eq:Compact_Heisenberg_Def}
\cheisenberg \equiv \left\langle \Hc^\alpha_n\forwhich \alpha\in\Delta,\, n \in \nats\right\rangle\,.
\end{align}
By the Lie bracket relation \eqref{eq:Compact_Subalg_Commu} a product of compact Heisenberg generators on the right of $\Ec_m^\alpha$ is in the universal enveloping algebra equivalent to a finite sum of products of elements of the compact Heisenberg algebra $U_j\in\Uc(\cheisenberg)$ to the left of different elements $\Ec_{n_j}^\alpha$
\begin{align}
\label{eq:Commutation_Ecn_Heisenberg}
&\Ec_m^\alpha U = \sum_{j} U_j \Ec_{n_j}^\alpha \in \Uc(\kfrak)\,.
\end{align}
The explicit form of $U_j$ is not important for the purpose of the paper. However, we sketch the evaluation of $U_j$ in Appendix \ref{app:Appendix_Details}.

Second, the Lie bracket of the maximal compact subalgebra is not additive in the loop index of the generators and does not provide a Borel subalgebra, but admits a filtered structure.
For example, the maximal compact subalgebra $\kfrak$ is \textit{not} of Kac-Moody type but admits infinite many ideals. These ideals are essential to construct finite dimensional representations of the maximal compact subalgebra in terms of the parabolic algebra \cite{Kleinschmidt:2021agj}.

\subsection{Parabolic algebra} 
\label{subsec:Parabolic_algebra} 
In this section we extend the (single) parabolic algebra in \cite{Kleinschmidt:2021agj} to a double parabolic algebra and a Lie algebra homomorphism from the maximal compact subalgebra to this double parabolic algebra. We prove a proposition on the form of the ideals of the maximal compact subalgebra and elaborate on a certain group action on the parabolic algebra.

To construct representations of the maximal compact subalgebra of an affine algebra, \cite{Kleinschmidt:2021agj} constructs a (single) parabolic algebra together with two Lie algebra homomorphisms from the maximal compact subalgebra to this parabolic algebra. The parabolic algebra allows for an algorithmic construction of representations. 
By the pull back of the Lie algebra homomorphism these are also representations of the maximal compact subalgebra. 
However, these representations do not exhaust the set of all representations of the maximal compact subalgebra. In particular, the basic representation cannot be projected to these representations. To construct a larger set of representations of the maximal compact subalgebra, we generalize the (single) parabolic algebra of \cite{Kleinschmidt:2021agj} to a double parabolic algebra $\kfrakpar$ together with a Lie algebra homomorphism from the maximal compact subalgebra to the double parabolic algebra. The general construction of representations of the double parabolic algebra is in Appendix \ref{app:Reps}, these representations are pulled back with the Lie algebra homomorphism to representations of the maximal compact subalgebra.
For the double parabolic algebra we prove two important propositions.

Let us define the double parabolic algebra.  
For the finite Lie algebra $\gfinite$ with compact subalgebra $\kfinite$ and non-compact orthogonal complement $\pfinite$, for $\reals\llbracket u^2\rrbracket$ the formal power series in $u^2$ and for $\mathcal{C}_2=\{1,r\}$ the group of order $2$, the \textit{double parabolic algebra} $\kfrakpar$ is 
\begin{align}
\label{eq:Double_Paraboic_Algebra}
\kfrakpar \eq 
\left(\kfinite \otimes \reals  \llbracket u^2\rrbracket\otimes \mathcal{C}_2 \right) \oplus \left(\pfinite \otimes u \reals\llbracket u^2\rrbracket\otimes \mathcal{C}_2\right).
\end{align}
The Lie bracket is multiplicative in the second and third factor of the tensor product and induced from $\gfinite$ on the first factor
\begin{align}
\label{eq:Comm_PQ}
\left[x\otimes u^k\otimes r^m, y\otimes u^l \otimes r^n \right] \eq  \left[x,y\right]\otimes u^{k+l}\otimes r^{m+n} .
\end{align}
For $k\in\nats_0$, $n\in\ints$ and $\alpha\in\Delta$, a parametrization of the double parabolic algebra is in terms of the generators\footnote{The generators which differ only by an even number in $n$ are equal, however, this definition allows to drop many $\modulo{}{2}$ symbols.}
\begin{align}
\label{eq:CommPQ}
P_{k,n}^\alpha\aeq \left(E^\alpha+(-1)^k E^{-\alpha}\right)\otimes u^{k}\otimes r^{n}\, ,\nn\\
\Qp^\alpha_{2k+1,n} \aeq H^\alpha\otimes u^{2k+1}\otimes r^{n}\,
\end{align}
with the Lie bracket in \eqref{eq:Comm_PQ}.
The abelian subalgebra of the generators
\begin{align}
\pheisenberg  \equiv \left\langle \Qp^\alpha_{2k+1,n}\forwhich \alpha\in\Delta,\, k\in\nats_0,\,0\leq n\leq 1\right\rangle
\end{align}
is called the \textit{parabolic Heisenberg algebra}, but analogous to the compact Heisenberg algebra it is \textit{not} a Heisenberg algebra.

We give some comments on the single factors of the double parabolic algebra.
The $\mathcal{C}_2$ factor induces a natural $\ints_2$ grading on the double parabolic algebra. The level $0$ subalgebra with respect to this grading has elements $x\otimes u^m\otimes 1$, which is the \textit{single parabolic algebra} in \cite{Kleinschmidt:2021agj}. It allows to extend most results for the single parabolic algebra in \cite{Kleinschmidt:2021agj} to the double parabolic algebra.

The second factor $\reals\llbracket u^2\rrbracket$, $u\reals\llbracket u^2\rrbracket$ has a graded structure in the power of $u$. This graded structure extends to the parabolic algebra and its universal enveloping algebra
\begin{align}
\kfrakpar \aeq \bigoplus_{l=0}^{\infty} {\kfrakpar}_l \, ,
\qquad
\Uc(\kfrakpar) \eq \bigoplus_{l=0}^{\infty} \Uc(\kfrakpar)_l\,.
\end{align}
The graded structure allows to identify ideals of the double parabolic algebra. For all $N\in\nats$, the double parabolic algebra and its universal enveloping algebra have the ideals
\begin{align}
\label{eq:IdealsIN}
\ideal_N \aeq \bigoplus_{l=N+1}^{\infty} {\kfrakpar}_l \, ,
\qquad
\ideal^\Uc_N \eq \bigoplus_{l=N+1}^{\infty} \Uc(\kfrakpar)_l\,.
\end{align}
The cosets of the double parabolic algebra and its universal enveloping algebra by these ideals are the quotient parabolic algebras
\begin{align}
\label{eq:QuotientParALg}
\kfrakpar^N \eq \frac{\kfrakpar}{\ideal_N}\, ,
\qquad 
\Uc(\kfrakpar)^N \eq \frac{\Uc(\kfrakpar)}{\ideal^\Uc_N}\,.
\end{align}

Now, let us produce the Lie algebra homomorphism from the double parabolic algebra to the maximal compact subalgebra. In \cite{Kleinschmidt:2021agj} a Lie algebra homomorphism from the (single) parabolic algebra to the maximal compact subalgebra is constructed.
We extend this construction to a Lie algebra homomorphism from the maximal compact subalgebra to the double parabolic algebra. Therefore, we use that the generators of the maximal compact subalgebra \eqref{eq:Generatingkfrakpar} are a subset of the loop generators of $\gfinite$ with variable $t$ on which we define a Möbius transformation
\begin{align}
\label{eq:MoebiusTrans}
t(u) = \frac{1-u}{1+u} \in\reals\llbracket u\rrbracket.
\end{align}
This Möbius transformation extends to a Lie algebra homomorphism $\rho$ from the maximal compact subalgebra to the parabolic algebra \cite{Kleinschmidt:2021agj}.\footnote{\label{ftn:Singel_Double_Parabolic} In \cite{Kleinschmidt:2021agj} Lie algebra homomorphisms $\rho_{\pm}$ from the maximal compact subalgebra to the single parabolic algebra are constructed from the Möbius transformation $t_\pm (u) = \frac{1\mp u}{1\pm u}$. Both homomorphisms map the maximal compact subalgebra differently to the single parabolic algebra. A tensor product of two representation of the single parabolic algebra which are pulled back by the different Lie algebra homomorphisms $\rho_{+}$ and $\rho_-$ to $\kfrak$-representations, is not a $\kfrak$-representation. The double parabolic algebra together with one Lie algebra homomorphism $\rho$ solves this.} 

\begin{proposition}[Extension of Proposition 5 in \cite{Kleinschmidt:2021agj}]
\label{prop:LieAlgHomo_rho}
For each $n\in\ints$ and $\alpha\in\Delta$, the map $\rho\, :\; \kfrak\rightarrow\kfrakpar$
\begin{align}
\label{eq:LieHomDouble}
\rho\left(\Ec^\alpha_{n}\right)
\aeq  \half \sum_{k= 0}^{\infty} a^{(n)}_{k} P^\alpha_{k,n}\, ,
\qquad 
\rho\left(\Hc^\alpha_{n}\right) \eq  \half\sum_{k= 0}^{\infty} a^{(n)}_{2k+1} \Qp^\alpha_{2k+1,n}\, .
\end{align}
is an injective Lie algebra homomorphism.
The coefficients $a_k^{(n)}$ are the coefficients of the power series expansion in \eqref{eq:MoebiusTrans} and are evaluated in \cite{Kleinschmidt:2021agj}
\begin{align}
a_{2k}^{(n)}= 2\sum_{\ell=0}^n\binom{2n}{2l}\binom{k-\ell+n-1}{k-\ell},
\quad 
a_{2k+1}^{(n)}= -2\,\sign(n) \sum_{\ell=0}^{n-1}\binom{2n}{2l+1}\binom{k-\ell+n-1}{k-\ell}.
\end{align}
The homomorphism $\rho$ respects the even/odd grading of the loop level of $\kfrak$.
\end{proposition}
\begin{rem}
\label{rem:Coef_ank_Polynomials}
The coefficients $a_k^{(n)}$ satisfy two identities. For fixed $n \in \nats$, the coefficients $a_k^{(n)}$ are polynomials in $k$ of degree $n-1$. For fixed $k\in\nats_0$, the coefficients $a_k^{(n)}$ are polynomials in $n$ of degree $k$. The expansions of the first few elements in $\Ec_n^\alpha$ is in \eqref{eq:LieHomDouble_Expl}
\end{rem}
\begin{rem}
For each $N\in\nats_0$, the codomain of $\rho$ can be quotiented by the ideal $\ideal_N$. This defines the surjective Lie algebra homomorphisms 
\begin{align}
\rho^N:\,\basic\,\rightarrow\, \kfrakpar^N
\end{align}
by \eqref{eq:LieHomDouble} but the summation from $k=0$ to $k=N$ resp. $k=\left\lfloor\frac{N}{2}\right\rfloor$ \cite{Kleinschmidt:2021agj}.
\end{rem}

\begin{proof}
The homomorphism $\rho$ respects the even-odd grading of the loop index of the maximal compact subalgebra because $r\neq1$, but $r^2=1$. 

Setting $r=\pm 1$ reduces the homomorphism $\rho$ to the injective Lie algebra  homomorphisms $\rho_{\pm}$ from the maximal compact subalgebra to the single parabolic algebra (see footnote \ref{ftn:Singel_Double_Parabolic}). Because $\rho_{\pm}$ are injective Lie algebra homomorphisms, also $\rho$ is an injective Lie algebra homomorphisms. 
\end{proof}

For each $N\in\nats_0$, the Lie algebra homomorphism $\rho$ allows to identify ideals in the maximal compact subalgebra $\kfrak$ as the inverse image of ideals $\ideal^N$ in the parabolic algebra. For these ideals, we use the same symbol $\ideal_N$.
While the ideals of the parabolic subalgebra derive immediately from the parabolic grading, the form of the ideals in the maximal compact subalgebra is more subtle. A generating identity for the ideals is subject of Proposition \ref{prop:IdealRelations}. In the proof of the proposition and also later in this work, we use an identity for sums of polynomials with binomial factors. Therefore, let us first give this identity globally and then use it to prove the proposition.

For $p\in\reals[m]_{N-1}$ a polynomial in $m$ of degree less then $N$, the binomial sum
\begin{align}
\label{eq:BinomialId}
\sum_{m=0}^{N}(-1)^m \binom{N}{m} p(m) \eq 0
\end{align}
vanishes. This follows by differentiating $k$ times for $1\leq k\leq N-1$ the binomial expansion of $(1+x)^N$ w.r.t. $x$ and then setting $x$ to $-1$.
With this identity, we prove the next proposition about ideals of the maximal compact subalgebra. A different formula for the same ideals is proposed in \cite{Kleinschmidt:2021eal}.

\begin{proposition}[Ideals]
\label{prop:IdealRelations}
For all $N\in\nats_0$, $\alpha\in\Delta$, $a\in\nats$, $b,n\in\ints$ the elements
\begin{align}
\label{eq:IdealRelations}
\sum_{m=0}^{N+a}(-1)^m\binom{N+a}{m}\,\Ec^\alpha_{n+2bm}\in \ideal_N\, ,\quad  \sum_{m=0}^{N+a}(-1)^m\binom{N+a}{m}\,\Hc^\alpha_{n+2bm} \in \ideal_N
\end{align}
of the filtered algebra are in the ideal $\mathcal{I}_N$.
\end{proposition}
\begin{proof}
For every $N\in\nats_0$, $a\in\nats$, $b,n\in\ints$  by identity \eqref{eq:BinomialId} and Remark \ref{rem:Coef_ank_Polynomials}, the elements \eqref{eq:IdealRelations} are in the kernel of $\rho^N$
\begin{align}
\rho^N\left(\sum_{m=0}^{N+a}(-1)^m\binom{N+a}{m}\,\Ec^\alpha_{n+2bm}\right) 
\aeq 
\half  \sum_{k= 0}^{N} P^\alpha_{k,n}\sum_{m=0}^{N+a}(-1)^m\binom{N+a}{m}\,  a^{(n+2bm)}_{k} \eq 0
\end{align}
The same holds for the second element in \eqref{eq:IdealRelations}.
\end{proof}     

This identity of ideals in the maximal compact subalgebra is important to investigate finite dimensional $\kfrakpar$-representations. But before we can do so, let us prove a proposition about the action of certain group elements on the parabolic algebra at the end of this section.

\begin{proposition}
\label{prop:WeylGrou}
For all $N\in\nats$ and $\alpha,\beta\in\Delta$, the group elements 
\begin{align}
\label{eq:WeylGroupElements}
\omega_{\alpha,N} \aeq \exp(\sum_{k\geq 0}\tfrac{2}{2k+1}\Qp^\alpha_{2k+1,0}) \, \in\, \Uc(\kfrakpar)^N\, 
\end{align}
act on the parabolic algebra $\kfrakpar^N$ by conjugation and satisfy
\begin{alignat}{4}
\label{eq:WeylGroup}
\omega_{\alpha,N}\rho^N\left(\Ec^\beta_n\right)\omega_{\alpha,N}^{-1}\eq  \rho^N\left(\Ec^\beta_{n-(\alpha|\beta)}\right)\, ,
\qquad 
\omega_{\alpha,N} \rho^N\left(\Hc^\beta_n\right)\omega_{\alpha,N}^{-1} \eq \rho^N\left(\Hc^\beta_n\right)\, . 
\end{alignat} 
\end{proposition}

\begin{rem}
First, the exponential series in \eqref{eq:WeylGroupElements} truncates at parabolic level $N$. Second, the conjugation group action is the exponentialized adjoint Lie algebra action $\adj$, which is used in the proof.
Third, the Lie algebra homomorphism $\rho$ resp. $\rho^N$ acts group like on the universal enveloping and thus on the group exponentiation of the maximal compact subalgebra. Then, the group element $\omega_{\alpha,\infty}$ (defined properly later) may be interpreted as the Lie algebra homomorphism acting on the group element $w_\alpha$ of $\gaffine$ which acts by
\begin{alignat}{4}
\label{eq:WeylGroup_Alg}
w_\alpha \Ec^\beta_nw_\alpha^{-1} \eq \,\Ec^\beta_{n-(\alpha|\beta)}\, ,
\qquad 
w_\alpha \Hc^\beta_n w_\alpha^{-1} \eq 
\Hc^\beta_n \, . 
\end{alignat} 
The relation of $\omega_{\alpha,N}$ to these group elements provides further inside in our construction but will not be used in this work.
\end{rem}
\begin{proof}
We use that conjugation is the group action associated to the adjoint representation of the Lie algebra. 
The parabolic Heisenberg algebra is abelian which proves the action of $\omega_{\alpha,N}$ on $\rho^N(\Hc_n^\beta)$. If $(\alpha|\beta) = 0$, $\Qp^\alpha_{2k+1,0}$ commutes with $P^\beta_{l,a}$ which proves the action of $\omega_{\alpha,N}$ on $\rho^N(\Ec^\beta_{n})$ for $(\alpha|\beta)=0$. 
 
For $(\alpha|\beta)\neq 0$ the adjoint map of the double parabolic algebra acts by
\begin{align}
\adj(\Qp_{2k+1,0}^\alpha)^m P_{l,a}^\beta = (\alpha|\beta)^m P_{l+(2k+1)m,a}^\beta\,.
\end{align}
For $n=0$ this gives
\begin{align}
\exp\left({\sum_{k\geq 0}\tfrac{2}{2k+1}\adj(\Qp^\alpha_{2k+1,0})}\right) P^\beta_{0,0}  \eq \sum_{k=0}^\infty c^{\alpha,\beta}_{k} P^j_{k,0} 
\end{align}
where
\begin{align}
c^{\alpha,\beta}_{k} = \frac{1}{k!}\left(\frac{\drm}{\drm x}\right)^k e^{\sum_{l\geq 0}\frac{2 (\alpha|\beta)}{2l+1}x^{2l+1}}\Big\vert_{x=0}\, . 
\end{align}
The coefficients $c^{\alpha,\beta}_{k}$ can be solved for $(\alpha|\beta) = \pm 1,\pm 2$. For $(\alpha|\beta) = \pm 2$ the coefficients are 
\begin{align}
c^{\alpha,\beta}_0 =1\qquad
c^{\alpha,\beta}_k = (\pm 1)^k 4k\, .
\end{align}
For $(\alpha,\beta) = \pm 1$ the coefficients $c_k^{\alpha,\beta}$ are
\begin{align}
c_0 = 1\qquad
c_k = (\pm 1)^k 2\, .
\end{align}
By \eqref{eq:LieHomDouble_Expl}, this proves the Proposition for $n=0$. For $n=-1$ and \eqref{eq:LieHomDouble_Expl}, the left hand side evaluates to
\begin{align}
\exp\left({\sum_{k\geq 0}\tfrac{2}{2k+1}\adj(\Qp^\alpha_{2k+1,0})}\right)\left(  P^\beta_{0,1} +2\sum_{k=0}^N (-1)^k  P^\beta_{k,1}\right) \eq \sum_{k=0}^\infty c^{\alpha,\beta}_k P^\beta_{k,1} 
\end{align}
with 
\begin{align}
c^{\alpha,\beta}_k = \frac{1}{k!}\left(\frac{\drm}{\drm x}\right)^k\left(1+2\sum_{l=1}^Nx^l\right) e^{\sum_{l\geq 0}\frac{2 (\alpha,\beta)}{2l+1}x^{2l+1}}\Big\vert_{x=0}\, . 
\end{align}
For $(\alpha,\beta) = \pm 2, \pm 1$ the coefficients are
\begin{align}
(\alpha,\beta) = + 2 \quad &\rightarrow \quad c^{\alpha,\beta}_0 = 1,\qquad c^{\alpha,\beta}_{k} =  2(1+2k^2),\nn\\
(\alpha,\beta) = + 1 \quad &\rightarrow \quad c^{\alpha,\beta}_0 = 1,\qquad c^{\alpha,\beta}_{k} =  4k, \nn\\
(\alpha,\beta) = - 1 \quad &\rightarrow \quad c^{\alpha,\beta}_0 = 1,\qquad
c_k = 0,\nn\\ 
(\alpha,\beta) = - 2 \quad &\rightarrow \quad c^{\alpha,\beta}_0 = 1,\qquad
c^{\alpha,\beta}_k = (- 1)^k 2\, .
\end{align}
With \eqref{eq:LieHomDouble_Expl}, this proves the Proposition for $n = 0$ and $n=-1$. 
Assuming the Proposition holds for $n\leq N\in\nats_0$, then by induction
\begin{align}
\exp\left({\sum_{k\geq 0}\tfrac{2}{2k+1}\adj(\Qp^\alpha_{2k+1,0})}\right) \Ec^\beta_{N+1} 
\aeq
\left[\Hc^\beta_1, \exp\left({\sum_{k\geq 0}\tfrac{2}{2k+1}\adj(\Qp^\alpha_{2k+1,0})}\right) \Ec^\beta_N\right]\nn\\
&+ \exp\left({\sum_{k\geq 0}\tfrac{2}{2k+1}\adj(\Qp^\alpha_{2k+1,0})}\right) \Ec^\beta_{N-1}
\nn\\
\aeq \Ec^\beta_{N+1-(\alpha|\beta)}
\end{align}
the proposition holds for all $N\in\nats_0$. 
The same induction argument applies to $-N\in\nats_0$. 
\end{proof}

In this section we introduced the maximal compact subalgebra, the double parabolic algebra and a Lie algebra homomorphism form the maximal compact subalgebra to the double parabolic algebra. Using the double parabolic algebra, we proved Proposition \ref{prop:IdealRelations} on the generating structure of the ideals $\ideal_N$ and for the double parabolic algebra we provided the action of a certain group element.

\section{Compact subalgebra on basic representation}
\label{sec:ActionOnBasic}
The previous two sections put us in a good place to discuss the action of the maximal compact subalgebra on the basic representation. Therefore, we prove a number of propositions in this section, which are important to analyze the $\kfrak$-structure of the basic representation in the next section. 
 
The maximal compact subalgebra acts on the basic representation by the induced action from the affine algebra. On maximal weights, the action of $\Ec_n^\alpha$ is given by
\begin{align}
\label{eq:Ec_nOnMaximalState}
\Ec_n^\alpha \, e^\gamma\otimes 1 \aeq \epsilon(\alpha,\gamma)\,e^{\gamma+\alpha} \otimes S_{-n-(\gamma,\alpha)-1}(\alphaSet)+\epsilon(\alpha,\gamma)\,e^{\gamma-\alpha} \otimes S_{n+(\gamma,\alpha)-1}(\{-\alpha\})\, .
\end{align}
Proposition \ref{prop:kOnMaximalStates} states two important properties for the action of the compact subalgebra on the basic representation. 

\begin{proposition}
\label{prop:kOnMaximalStates}
The action of the compact subalgebra $\kfrak$ on a maximal state generates the entire basic representation. The action of the compact Heisenberg algebra on a maximal state generates all states with the same $\gfinite$ weight:
\begin{itemize}
\item[a.)] $\qquad \forall\, X \in \basic,\, \gamma \in \Qr \;  \exists\, U\in\Uc(\kfrak)\,
: \; X = U\left( e^\gamma \otimes 1\right)$ 
\item[b.)] $\qquad \forall\, \gamma\in \Qr,\, f\in\reals[\{\alpha\in\Delta\}]\; \exists\, \Fc^{\Hc}_f\in\Uc(\cheisenberg)\, 
:
\; e^\gamma \otimes f= \Fc^{\Hc}_f\left( e^\gamma\otimes 1\right)$.
\end{itemize}
\end{proposition}
\begin{proof}
First, we show that every maximal state can be mapped to the highest state by repeated action of $\kfrak$
\begin{align}
\label{eq:MaxState_ToHeighestState}
\forall\, \gamma\in \Qr \;  \exists\, U'\in\Uc(\kfrak): \; e^0\otimes 1 = U'\left(e^\gamma\otimes 1\right).
\end{align}
W.l.o.g. assume $\gamma = \sum_{i=1}^r k_i \alpha^i$ with $k_i\geq 0$. Choose $1\leq j\leq r$ s.t. $k_j\neq 0$ and $n\eq -(\gamma|\alpha^j)+1<0$. Such $j$ exists because $(\cdot|\cdot)$ is positive definite. 
Then, by equation \eqref{eq:Ec_nOnMaximalState} 
\begin{align}
\Ec_n^{\alpha^j}\,  e^\gamma\otimes 1 = \epsilon(\gamma,\alpha^j)  e^{\gamma-\alpha^j}\otimes 1. 
\end{align}
The procedure repeats for $\gamma-\alpha^j$ and iteratively maps $e^{\gamma}\otimes 1\in\basic$ onto the highest-weight $e^0\otimes 1$. It works analogous for $k_i\in\ints$. This shows \eqref{eq:MaxState_ToHeighestState}.

By induction in the loop level of a state $X\in\basic$, we show that the basic representation is generated by the action of $\kfrak$ on the highest-weight $\basic = \Uc(\kfrak)\,e^0\otimes 1$. Assume that for all $X\in\basic$ with $\drm X = N'X$ for $N'\leq N\in\nats$ holds $X\in\Uc(\kfrak)\, e^0\otimes 1$.

Let $Y\in\basic$ s.t. $\drm Y = (N+1)Y$. Because $\basic$ is generated by the action of $\{T_n\in\gaffine\forwhich n\leq -1\}$ on the highest state it exists an $Y'\in\basic$, $n\in\nats$ with $\drm Y'=(N-n+1)Y'$ s.t. $Y = T_{-n} Y'$. Then, by assumption
\begin{align}
Y \eq (T_{-n}+\tau(T_{-n})) Y' - \taufinite(T)_nY' \,\in\, \Uc(\kfrak)\,e^0\otimes 1
\end{align} 
This proves \textit{a}. To prove \textit{b}, we use that $\Set{H_n^\alpha\forwhich n\leq -1,\, \alpha\in\Delta}$ on a maximal state spans all states with the same $\gaffine_0$. Then, the same induction proves \textit{b}.
\end{proof}

By Proposition \ref{prop:kOnMaximalStates}, for a state $e^\gamma\otimes f(\{\alpha\in\Delta\})$ exists generating elements $\Fc^{\Hc}_f\in\Uc(\cheisenberg)$, which generates this state from the action on the maximal state $e^\gamma\otimes 1$. By the very definition of the generating elements, they satisfy for each $n\in\nats$ and $\alpha \in \Delta$ 
\begin{align}
\label{eq:HFf=FHf}
\Hc_n^\alpha \Fc^{\Hc}_f \eq \Fc^{\Hc}_{\Hc_n^\alpha f} \,.
\end{align}

Let us also introduce a short notation for the generating element $\Fc^{\Hc}_{S_n}$ of Schur polynomials, because the Schur polynomials are essential in the vertex operator algebra. The generating elements of the Schur polynomials are called \textit{$\Hc$-Schur polynomials} and are defined by
\begin{align}
\label{eq:H_SchurPolynomial_Def}
S^{\Hc}_n(\{\Hc^\alpha\}) = \Fc^{\Hc}_{S_n(\{\alpha\})} \in\Uc(\cheisenberg).
\end{align}
The action of a $\Hc$-Schur polynomial on a maximal state generates the state with the same $\gfinite$ weight and the respective Schur polynomial.
The $\Hc$-Schur polynomials satisfy a recursion identity
\begin{align}
\label{eq:SchurInductionFormula}
S^{\Hc^\alpha}_N(\{\Hc^\alpha\}) \eq  \frac{2}{N}\sum_{n=1}^N \left(\Sh_{N-2n}(\{\Hc^\alpha\})-\Hc_n S_{N-n}(\{\Hc^\alpha\}) \right)\,  
\end{align}
which derives from adding and subtracting $H^\alpha_{n}$ to
\eqref{eq:SchurId} and which is used for proofs by induction.
We prove the next lemma by induction on this identity. The lemma provides a closed expression for the $\Hc$-Schur polynomials and will be used to prove the last proposition in this section.
\begin{lemma}
\label{lem:SchurH}
For every $\alpha\in\Delta$, the $\Hc$-Schur polynomials are given in the expansion
\begin{align}
\label{eq:SchurH}
S_{n}^{\Hc}(\{\Hc^\alpha\}) \eq \sum_{m=0}^{\left\lfloor\frac{n}{2}\right\rfloor} 
\exp{\left(\sum_{l\geq 1} -\tfrac{2}{l}\Hc_l^\alpha x^l\right)}\Big|_{x^{n-2m}}\,.
\end{align}
\end{lemma}
\begin{rem}
\label{rem:SchurInHc}
Evaluating the generating series in Lemma \ref{lem:SchurH} gives the explicit form in terms of the sets $\mathcal{K}_n \eq \Set{(k_1,\ldots k_n) \in \nats^n_0\forwhich \sum_{l=1}^n lk_l = n}$ 
\begin{align}
\label{eq:SchurInHc}
S_{n}^{\Hc}(\{\Hc^\alpha\}) \aeq \sum_{m=0}^{\left\lfloor\frac{n}{2}\right\rfloor} \sum_{k \in \mathcal{K}_{n-2m}} \prod_{l=1}^{n-2m}\frac{1}{k_l!}\left(-\frac{2}{l}\Hc^\alpha_l\right)^{k_l}\\
\aeq\sum_{M=0}^n 
\sum_{m_1=1}^{n}\sum_{m_2=1}^{n-m_1}\ldots 
\sum_{m_{M}=1}^{n-m_1-\dots-m_{M-1}}
r_n(m_1,\ldots,m_M)\,\prod_{i=1}^{M}\left(-\frac{2}{m_i}\Hc_{m_i}^\alpha\right)\nn
\end{align}
where $r_n(m_1,\dots,m_M)$ is the factorial factor from the exponential. If $k_n$ is the multiplicity of the integer $n$ appearing in the set $\{m_1,\dots, m_M\}$ then, $r$ is given by
\begin{align}
r_n(m_1,\dots,m_M)\aeq
\begin{cases}
\prod_{i\geq 1}\frac{1}{k_i!} & \textrm{if }n+\sum_{i=1}^M m_i \ \textrm{is even}\\
0 & \textrm{if }n+\sum_{i=1}^M m_i \ \textrm{is odd}\, .
\end{cases}
\end{align}
\end{rem}

\begin{proof}
We prove the expansion in \eqref{eq:SchurInHc} by induction using \eqref{eq:SchurInductionFormula}.
For $n = 0,1$ the identity holds since for $\gamma \in \Qr$
\begin{align}
S_0^{\Hc}(\{\Hc^\alpha\})e^{\gamma}\otimes 1 = e^{\gamma}\otimes 1\qquad
S_1^{\Hc}(\{\Hc^\alpha\})e^{\gamma}\otimes 1 = -2\Hc_1^\alpha e^{\gamma}\otimes 1 = e^{\gamma}\otimes S_1(\{\alpha\})\,. 
\end{align}
For $k\in\Kc_M$ set $d_k = \prod_{l=1}^{M} \frac{1}{k_l!}(-\tfrac{2}{l})^{k_l}$ and assume \eqref{eq:SchurInHc} is true for $N'\leq N\in\nats_0$. We find
\begin{align}
\label{eq:ProofLemSH}
\Sh_{N+1}(\{\Hc^\alpha&\}) 
\eq \frac{2}{N+1}\sum_{n=1}^{N+1} \left(\Sh_{N-2n+1}(\{\Hc^\alpha\})-\Hc^\alpha_n\Sh_{N-n+1}(\{\Hc^\alpha\}) \right)\nn\\
\aeq\sum_{p\geq 0}\sum_{k\in\Kc_{p}}\delta_{\modulo{(p+N+1)}{2},0}\sum_{n=1}^{N+1} \frac{2}{N+1}\left(\Theta(N-2n+1-p) 
 -\frac{nk_n}{2}\right)d_k\prod_{l=1}^p(\Hc^\alpha_l)^{k_l}
\nn\\
\aeq
\sum_{p\geq 0}\sum_{k\in\Kc_{p}}\delta_{\modulo{(p+N+1)}{2},0}\Theta(N+1-p) d_k\prod_{l=1}^p(\Hc^\alpha_l)^{k_l}\, ,
\end{align}
where $\Theta$ is the Heaviside function.
This proves the Lemma. A detailed calculation is in Appendix \ref{app:Appendix_Details}. 
\end{proof}
Lemma \ref{lem:SchurH} provides an explicit formula for the $\Hc$-Schur polynomials. The Lie algebra homomorphism $\rho^N$ maps the $\Hc$-Schur polynomials to elements in the universal enveloping of the parabolic Heisenberg algebra. We call these elements the $\Qp$-Schur polynomials. For the $\Qp$-Schur polynomials we prove an important proposition. However, in order to do so, we first prove a lemma and then define the $\Qp$-Schur polynomials properly.
\begin{lemma}
\label{lem:p_Weyl}
For each $N\in\nats_0$, the elements $p^\alpha_{N,0}$ and $p^\alpha_{N,1}$ defined by 
\begin{multicols}{2}
\noindent
\begin{alignat}{4}
&p^\alpha_{N,0}(n)\;\colon\ &&2\ints\ &&\rightarrow \quad  &&\Uc(\pheisenberg)\nn\\
& && \ n&& \mapsto && -\frac{2}{n}\rho^N(\Hc^\alpha_n)\, ,\nn
\end{alignat}
\begin{alignat}{4}
\label{eq:p(n)_Definition}
&p^\alpha_{N,1}\;\colon\ &&2\ints+1\ &&\rightarrow \quad  &&\Uc(\pheisenberg)\nn\\
& && \quad n&& \mapsto && -\frac{2}{n}\rho^N(\Hc^\alpha_{n})\, ,\nn
\end{alignat}
\end{multicols}
\noindent
are polynomials in $n$ taking values in the parabolic Heisenberg algebra. At $n=0$, the polynomial $p^\alpha_{N,0}(0)$ satisfies
\begin{align}
p^\alpha_{N,0}(0) \eq \prod_{l=0}^{\left\lfloor\frac{N}{2}\right\rfloor} \frac{4}{2l+1}\Qp_{2l+1,0}^\alpha\, .
\end{align}
\end{lemma}
\begin{rem}
\label{rem:p_Weyl}
Therefore, for every $N\in\nats_0$, the group element $(\omega_{\alpha,N})^2$ in \eqref{eq:WeylGroupElements} is
\begin{align}
(\omega_{\alpha,N})^2\eq \sum_{m\geq 0}\frac{1}{m!} \left(p^\alpha_{N,0}(0)\right)^m\, . 
\end{align}
\end{rem}
\begin{proof}
By Remark \ref{rem:Coef_ank_Polynomials}, $\rho^N(\Hc^\alpha_m)$ is a polynomial on the domain $m\in\nats$ taking values in the parabolic algebra. Because $\Hc^\alpha_0=0$, the polynomial $\rho^N(\Hc^\alpha_{2m})$ has a zero at $m=0$. Therefore \eqref{eq:p(n)_Definition} are polynomials as well. 
To evaluate $p^\alpha_{N,0}$ at zero, the Möbius transformation \eqref{eq:MoebiusTrans} is useful 
\begin{align}
\label{eq:proof_pl0}
p^\alpha_{N,0}(0)\Big|_{\Qp_{2l+1,0}^\alpha}= -\frac{1}{m}\rho^N(\Hc_{2m})\Big|_{\Qp_l} \eq - \frac{1}{m} \left(t(u)^{2m}-t(u)^{-2m}\right) \Bigg|_{m=0}\Bigg|_{u^{2l+1}} = \frac{4}{2l+1}\, .
\end{align}
\end{proof}
Now, let us define the $\Qp$-Schur polynomials and the generating elements $\Fc_f^{N,\Qp}$.
For each $N\in\nats_0$ and $f\in\reals[\{\alpha\in\Delta\}]$, the generating elements $\Fc_f^{\Hc}$ are mapped by the Lie algebra homomorphism $\rho^N$ group like to the quotient universal enveloping of the parabolic algebra defined  analogous to \eqref{eq:QuotientParALg} 
\begin{align}
\Fc^{N,\Qp}_f(\QSet) = \rho^N(\Fc^{\Hc}_f(\HSet)) \in \Uc(\pheisenberg)^N\, .
\end{align}
These elements are important to understand $\kfrakpar$-subrepresentation of the basic representation which justifies the additional notation. 
The image of $\rho^N$ on the $\Hc$-Schur polynomials are called \textit{$\Qp$-Schur polynomial}
\begin{align} 
S^{N,\Qp}_n(\{\Qp^\alpha\}) = \rho^N(\Sh_n(\{\Hc^\alpha\})))\in \Uc(\pheisenberg)^N\, .
\end{align} 
By Remark \ref{rem:Coef_ank_Polynomials}, Lemma \ref{lem:SchurH} and Remark \ref{rem:SchurInHc}, the coefficients of $\Uc(\pheisenberg)^N$ in $S^{N,\Qp}_{2n}$ and $S^{N,\Qp}_{2n+1}$ are polynomials in $n\in\nats_0$ of degree equal to the parabolic level of the product of $\Qp^\alpha_{l,a}$. In particular, the maps
\begin{multicols}{2}
\noindent
\begin{alignat}{4}
&\bar S_{\alpha,0}^{N,\Qp}\;\colon\ &&2\nats_0\ &&\rightarrow \quad  &&\Uc(\pheisenberg)^N \nn\\
& && \ n&& \mapsto && S^{N,\Qp}_{n}(\{\Qp^\alpha\})\, ,\nn
\end{alignat}
\begin{alignat}{4}
&\bar S_{\alpha,1}^{N,\Qp}\;\colon\ &&2\nats_0+1\ &&\rightarrow \quad  &&\Uc(\pheisenberg)^N\nn\\
& && \qquad n&& \mapsto && S^{N,\Qp}_{n}(\{\Qp^\alpha\})
\end{alignat}
\end{multicols}
\noindent
are polynomials in $n$ of degree $N$. These polynomials extend naturally to polynomials with the domain extended from $\nats_0$ to $\ints$. While for $n\in\nats$ and $a\in\{0,1\}$, the $\Qp$-Schur polynomials with negative index $S^{N,\Qp}_{-n}=0$ vanish, the extension $\bar S^{N,\Qp}_{a}({-n})$ can be non zero.

We use this definition of \textit{$\Qp$-Schur polynomials} together with Lemma \ref{lem:SchurH} and Lemma \ref{lem:p_Weyl} to prove the next proposition about the extended $\Qp$-Schur polynomials.

\begin{proposition}
\label{prop:ExtendedSchurPoly}
For all $n,N\geq 0$ and $\alpha\in\Delta$, $a\in{0,1}$, the extended $\Qp$-Schur polynomials $\bar S_{\alpha,a}^{N,-\Qp} \in\reals[n]$ satisfy 
\begin{align}
\bar S_{\alpha,a}^{N,-\Qp}(-n) = \bar S_{\alpha,a}^{N,\Qp}(n-2) \omega^2_{\alpha}\, .
\end{align}
\end{proposition}
\begin{proof}
The statement is independent of $\alpha\in\Delta$, $a\in\{0,1\}$ and $N\in\nats_0$. Therefore, for the proof, let us drop these indices on the elements of the parabolic Heisenberg algebra, on the $\Qp$-Schur polynomials and on the group element $\omega_{\alpha,N}$. 

Here, let us prove the statement in detail for $n$ even, the proof works the same for $n$ odd and is described at the end.
For the proof, we evaluate the coefficients of elements of $\Uc(\pheisenberg)$ in $\bar{S}^{-\Qp}(-2n)$. To specify these elements in $\Uc(\pheisenberg)$ let $L,L'$ be two finite sets of odd natural numbers and let
\begin{align}
\label{eq:QLQtLDefinition}
\Qp^L = \prod_{l\in L} \Qp_{l',0}\qquad
\Qt^{L'} = \prod_{l\in L'} \Qp_{l',1}\, .
\end{align}
In this notation, we evaluate the coefficient of the product $\Qp^L\Qt^{L'}$ in $\bar S(-2n)$. The element $\bar S(-2n)$ is the polynomial extension from the $\Qp$-Schur polynomials with $n\in\nats$ to $n\in\ints$.
Therefore, by the identity for generalized sums in Appendix \ref{app:PolySum}, $\bar S(-2n)$ has the form of $S(2n)$ but in terms of the generalized sums $\sum \rightarrow \gsum{}{}$ and $n\rightarrow -n$. Therefore, with Lemma \ref{lem:SchurH} and Lemma \ref{lem:p_Weyl} and $s = |L|$, $t=|L'|$ where $t$ is even since $n$ is even, the coefficient of $\Qp^L\Qt^{L'}$ in $\bar S(-2n)^{-\Qp}$ is
\begin{align}
\label{eq:PropCompleteSum}
\bar S^{-\Qp}(-2n)&\Big|_{\Qp^L\Qt^{L'}} 
\eq
(-1)^{s} 
\gsum{m_1=1}{-n}
\gsum{m_2=1}{-n-m_1}
\ldots
\gsum{m_{s+j}=1}{-n-m_1-\dots-m_{s+j-1} + \left\lfloor\frac{j}{2}\right\rfloor}
\ldots
\gsum{m_{s+t}=1}{-n-m_1-\dots-m_{s+t-1} + \frac{t}{2}}
\nn\\ 
&\ r(2m_1,\dots,2m_s,2m_{s+1}+1,\dots,2m_{s+t}+1)\,\prod_{i=1}^{s}p(2m_i)\prod_{i=s+1}^{s+t}p(2m_i-1)\, .
\end{align}
In the summation $\left\lfloor\frac{j}{2}\right\rfloor$ accounts for the $j-1$ number of sums with $s+1\leq i\leq s+t-1$.

Next, we replace the generalized sum with upper limit higher than the lower limit by \eqref{eq:Generalized_Sum_Id}. Therefore, first we focus on the sums $1\leq i\leq s$. The replacement also inverts $m_i\rightarrow -m_i$, but because $p(-2m_i) = p(2m_i)$ it is sufficient to replace the sums by
\begin{align}
\label{eq:Prf_Prop_GenS_Repl_Even}
\gsum{m_i=1}{-n-m_1-\dots -m_{i-1}} \ \rightarrow \ - 
\sum_{m_i=1}^{n-m_1-\dots -m_{i-1}-1} \; -\; \delta_{m_i,0}\,.
\end{align}

Second, we replace the sums $s+1\leq i\leq t$ pairwise with the pairs $i=s+2j-1$ and $i=s+2j$ by \eqref{eq:Generalized_Sum_Id}. Doing this for all sums simultaneously implies also $m_i\rightarrow -m_i$ in the summands and provides 
\begin{align}
\label{eq:Prf_Prop_GenS_Repl_Odd_1}
&\gsum{m_{s+2j-1}=1}{-n-m_1-\dots -m_{s+2j-2}+j-1}\quad
\gsum{m_{s+2j}=1}{-n-m_1-\dots -m_{s+2j-1}+j} 
p(2m_{s+2j-1}-1)p(2m_{s+2j}-1)
\nn\\[2pt]
&\ \rightarrow \ 
\sum_{m_{s+2j-1}=0}^{n-m_1-\dots -m_{s+2j-2}-j+1}
\quad 
\sum_{m_{s+2j}=0}^{n-m_1-\dots -m_{s+2j-1}-j} 
p(-2m_{s+2j-1}-1)p(-2m_{s+2j}-1)\,.
\end{align}
In the first sum in the last expressions, the summand $m_{s+2j-1} = n-m_1-\dots-m_{2+2j-1}-j+1$ gives for the second sum the limits $0\leq m_{s+2j} \leq -1$ which vanishes by \eqref{eq:Generalized_Sum_Id_0} and this term can therefore be removed from the first sum. Shifting all sums with $s+1\leq i\leq s+t$ by $m_i\rightarrow m_i -1$ to start at $m_{i}=1$ and using $p(-m) = p(m)$ results in the replacement of \eqref{eq:Prf_Prop_GenS_Repl_Odd_1} by
\begin{align}
\label{eq:Prf_Prop_GenS_Repl_Odd}
\sum_{m_{s+2j-1}=1}^{n-m_1-\dots -m_{s+2j-1}+j-2}\
\sum_{m_{s+2j}=1}^{n-m_1-\dots -m_{s+2j}+j-1} 
p(2m_{s+2j-1}-1)p(2m_{s+2j}-1)\,.
\end{align}

We insert the replacements in \eqref{eq:Prf_Prop_GenS_Repl_Even} and \eqref{eq:Prf_Prop_GenS_Repl_Odd} in \eqref{eq:PropCompleteSum}. The sum in \eqref{eq:Prf_Prop_GenS_Repl_Even} and the sum in \eqref{eq:Prf_Prop_GenS_Repl_Odd} will contribute to $\bar S^{\Qp}(2N-2)$, the Kronecker delta terms in \eqref{eq:Prf_Prop_GenS_Repl_Even} sum up to $\omega^2$.
To see this, we collect all terms with the same number of Kronecker delta terms $\delta_{0,m_i}$. The terms with $s'\leq s$ number of Kronecker delta terms $\delta_{m_i,0}$ for $i\leq s$ appear with multiplicity choose $s'$ out of $s$. 
Therefore, this rearrangement of sums provides us with 
\begin{align}
\bar S^{-\Qp}(-2n)\Big|_{\Qp^L\Qt^{L'}}
\aeq
\sum_{s'=1}^s\binom{s}{s'}
\sum_{m_{s'+1}=1}^{n-1}\ldots
\sum_{m_{s+j}=1}^{n-1-\dots-m_{s'+j-1}+\left\lfloor\frac{j}{2}\right\rfloor}
\ldots 
\sum_{m_{s+t}=1}^{n-1-\dots-m_{s'+t-1}+\frac{t}{2}}
\nn\\
& r(m_{s'+1},\dots,m_{s+t})\prod_{i=s'+1}^{s}p(2m_i)
\prod_{i=p+1}^{s+t}p(2m_i-1) 
\frac{1}{s'!}\prod_{i=1}^{s'}p(0)\Bigg|_{\Qp^L\Qt^{L'}} 
\end{align}
where a second factor of $(-1)^{s}$ appears from \eqref{eq:Prf_Prop_GenS_Repl_Even}.
The factorial $\frac{1}{s!}$ comes from the $s'$ number of arguments with same $m=0$ in the argument of $r$. The overall projection to $\Qp^L\Qt^{L'}$ is the sum of all possible projections of the different factors. Because $p(2m)$ is a series in $\Qp_{l,0}$ and $p(2m+1)$ is a series in $\Qp_{l,1}$, the sum over all different projections is the sum over all subset decompositions of $L = L(s')\cup L(s-s')$, where $L(s')$ has $s'$ number of elements fixed
\begin{align}
&\prod_{i=s'+1}^{s}p(2m_i)\prod_{i=s+1}^{s+t}p(2m_i-1) \frac{1}{s'!}\prod_{i=1}^{s'}p(0)\Bigg|_{\Qp^L\Qt^{L'}} \nn\\
&\eq
\sum_{L(s-s')\cup L(s')}\left(\prod_{i=s'+1}^{s}p(2m_i)\right)\Big|_{\Qp^{L(s-s')}}\left(\prod_{i=s+1}^{s+t}p(2m_i-1)\right)\Big|_{\Qt^{L'}} \frac{1}{s'!}\left(p(0)\right)^{s'}\Big|_{\Qp^{L(s')}}.
\end{align}
By Lemma \ref{lem:p_Weyl} and Remark \ref{rem:p_Weyl}, the last product with factors $p(0)$ projected on $\Qt^{L'}$ is $\omega^2$ projected on $\Qt^{L'}$. For the other factors of $p(m)$ with $m\neq 0$ we use Lemma \ref{lem:p_Weyl} and pull $\rho$ in front of the entire expression. Then, by \eqref{eq:SchurInHc} equation \eqref{eq:PropCompleteSum} becomes
\begin{align}
\bar S^{-\Qp}&(-2n)\Big|_{\Qp^L\Qt^{L'}} \nn\\
\aeq
\sum_{s'=1}^p \binom{s}{s'}
\sum_{L(s-s'),\, L(s')}
\rho^N\Bigg[
\sum_{m_{s'+1}=1}^{n-1}
\dots
\sum_{m_{s+j}=1}^{n-1-m_{s+j-1}-\left\lfloor\frac{j}{2}\right\rfloor}
\ldots 
\sum_{m_{s+t}=1}^{n-1-m_1-\dots -m_{s+t-1}+\frac{t}{2}}
\nn\\[2pt]
&\quad r(m_{s'+1},\dots,m_{s+t})\left(\prod_{i=s'+1}^{p}\frac{2}{2m_i}\Hc_{2m_i}\prod_{i=s+1}^{s+t}\frac{2}{2m_i-1}\Hc_{2m_i-1}\right) \Big|_{\Qp^{L(s')}\Qt^{L'}}\Bigg] \omega^2 \Big|_{\Qp^{L(s')}}\nn\\[2pt]
\aeq \sum_{s'=1}^p\binom{s}{s'}\sum_{L(s-s'),\, L(s')}\bar S(2(n-1))\Big|_{\Qp^L(s-s')\Qt^{L'}}\omega^2\Big|_{\Qp^L(s')}\eq \bar S(2(n-1))\omega^2\Big|_{\Qp^L\Qt^{L'}}\,.
\end{align}
The proof works analogously for $\bar S^{-\Qp}(-2n-1)$. Then $t$ is odd, and one uses the pair wise replacement of all sums in \eqref{eq:Prf_Prop_GenS_Repl_Odd} for $s+2\leq i\leq s+t$. The sum $s+1$ is replaced individually by \eqref{eq:Prf_Prop_GenS_Repl_Odd} where the upper limit of the sum reduces by $1$ because of \eqref{eq:Generalized_Sum_Id_0}. Then the same steps as for the even case provides $\bar S^{-\Qp}(-2n-1) =\bar S^{-\Qp}(2n-1)\omega^2$. 
\end{proof}

\begin{corollary}
\label{cor:S(-1)}
The extended $\Qp$-Schur polynomials $\bar S^{N,\Qp}_\alpha(n)$ vanish at $n=-1$
\begin{align}
\bar S^{N,\Qp}_\alpha (-1) = 0\,.
\end{align}
\end{corollary}
\begin{proof}
This follows from \eqref{eq:PropCompleteSum} for odd $n$ and \eqref{eq:Generalized_Sum_Id_0}. It is consistent with Proposition \ref{prop:ExtendedSchurPoly}.
\end{proof}

With the results of this section we are well equipped to analyze the $\kfrak$-structure of the basic representation, which is the subject of the next section.

\section{$\kfrak$-structure of basic representation}
\label{sec:CompleteDecomposition}
In this section we prove Theorem \ref{the:CompleteDecomposition}. For instance we give all $\kfrak$-subrepresentations in the basic representation such that the quotient of the basic representation by the subrepresentation is a finite dimensional $\kfrakpar$-representation. This is the finest $\kfrak$-structure of the basic representations and it allows us to provide the cosocle filtration as an infinite composition series of infinite dimensional $\kfrak$-subrepresentations of the basic representation.

First, we provide an infinite set of finite dimensional $\kfrak$-representations $\verQ(N)$\footnote{We chose the name $\verQ$ since the modules are truncated \textit{Verma} module of the parabolic Heisenberg algebra with generators $\Qp$.} and surjective projections $\Gmap_N :\,\basic \to \verQ(N) \simeq \basic/W_N$, which commute with the action of $\kfrak$ and project the basic representation on the finite dimensional $\kfrak$-representations $\verQ(N)$. This provides an infinite set of $\kfrak$-subrepresentation $W_N=\mathrm{Ker}(G_N)$ of the basic representation as the kernel of the projections.\footnote{\label{ftn:Embeding_Projection}\textbf{Embedding and projection.} For two vector spaces $V,W$ and a homomorphism $\Gmap':\,V\to W$ we say that $V$ is \textit{embedded }in $W$ if $\Gmap'$ is injective and if $\Gmap'$ is surjective we say that $W \simeq V/(\mathrm{Ker}(\Gmap')$ is a \textit{projection} of $V$. If the vector spaces are representations of a Lie algebra $\mathfrak{l}$ and $\Gmap'$ respects the action of $\mathfrak{l}$ we call the embedding (resp. projection) also a $\mathfrak{l}$-embedding (resp. $\mathfrak{l}$-projection). If it is clear from the context we also avoid the symbol $\mathfrak{l}$ and simply write embedding (resp. projection).}

Second, we prove that if the basic representation projects $\kfrak$-covariant on a finite dimensional $\kfrakpar$-representation, there exists an $N\in\nats_0$ such that the $\kfrakpar$-representations is equivalent to $\verQ(N)$ or to a quotient of $\verQ(N)$ by a subrepresentation of $\verQ(N)$.
This is equivalent to the statement that every $\kfrak$-subrepresentation $W\subset \basic$, for which the quotient $\basic/W$ is a finite dimensional $\kfrakpar$-representation by $\rho$, is $W_N$ for some $N\in\nats_0$ or a subrepresentation of $W_N$. Therefore, we may call this set of subrepresentations the finest $\kfrakpar$-structure (and probably the finest $\kfrak$-structure).

Third, we show that the infinite family of $\kfrak$-subrepresentations $(W_N)_{N\in\nats_0}$ is an infinite composition series with cosocle filtration (see Appendix \ref{app:Cosocle_Filtration}). In particular, the finite dimensional quotients $W_{N}/W_{N+1}$ are maximal semisimple resp. $W_{N+1}$ is the radical of $W_N$ and the inverse limit of the chain of $\kfrak$-subrepresentations is trivial.

In Subsection \ref{subsec:Embedding_k_rep} we derive the infinite chain of $\kfrak$-subrepresentations with finite co-dimension which is the cosocle filtration of $\basic$ and we derive the $\kfrak$-projection of the basic representation onto infinitely many finite dimensional $\kfrak$-representations. In Subsection \ref{subsec:Infinite_Limit} we show that the inverse limit of the $\kfrak$-subrepresentations is trivial and therefore the cosocle filtration is an infinite composition series of $\basic$ under $\kfrak$.

\subsection{$\kfrak$-subrepresentations in basic representation}
\label{subsec:Embedding_k_rep}
In this section we provide \textit{all} infinitely many projections of the basic representation onto finite dimensional $\kfrakpar$-representations (See Footnote \ref{ftn:Embeding_Projection} or Appendix \ref{app:Cosocle_Filtration}). Equivalently, we provide \textit{all} infinitely many $\kfrak$-subrepresentations of $\basic$, for which the quotient of $\basic$ by the subrepresentation is a finite dimensional $\kfrakpar$-representation. This set of $\kfrak$-subrepresentations also contains the cosocle filtration of the basic representation.

First, for all $N\in\nats_0$ we construct a finite dimensional $\kfrakpar$-representation $\verQ(N)$, which is also a representations of $\kfrak$ by pulling back with the Lie algebra homomorphism $\rho^N$. Second, we prove that the basic representation can be projected $\kfrak$-covariant onto $\verQ(N)$ for each $N\in\nats_0$ and we give the explicit $\kfrak$-projection map $\Gmap_N$. Third, we show that this infinite set of $\kfrak$-representations and the quotient of these representations by subrepresentations are all finite dimensional $\kfrakpar$-representations on which the basic representation projects with the pull back of $\rho$.  
Fourth, in a corollary we give the infinite descending chain of $\kfrak$-subrepresentations in the basic representation with cosocle filtration.

For each $N\in\nats_0$, we define a representation $\verQfull(N)$ as the tensor product of the universal enveloping algebra of the parabolic Heisenberg algebra $\Uc(\pheisenberg)$ quotiented by the ideal $\ideal^{\Uc}_N$ in \eqref{eq:IdealsIN} and the group algebra of the root lattice with the inclusion $\alpha\mapsto e_\alpha$  
\begin{align}
\label{eq:verQfull}
\verQfull(N) \eq  \frac{\Uc(\pheisenberg)}{\ideal^{\Uc}_N}\otimes\reals[\Qr ]\,.
\end{align}
This definition is reminiscent to the definition of the basic representation but in parabolic generators. In particular, for a fixed $\be_\alpha\in\reals[\Qr ]$, the representation $\verQfull(N)$ restricts to a finite dimensional quotient of a Verma module of the parabolic Heisenberg algebra.

The exponential series $\omega_{\alpha,N}$ in Proposition \ref{prop:WeylGrou} acts by multiplication on $\verQfull(N)$. This is well-defined, because first, only finitely many terms in the exponential series in \eqref{eq:WeylGroupElements} act non-trivially and second, they are given in terms of the parabolic Heisenberg algebra. Next, we consider the group which is generated by $\omega_{\alpha,N}^2$ for $\alpha\in\Delta$. This induces an equivalence relation on $\verQfull(N)$ which is generated by
\begin{align}
\label{eq:Equi_Relation}
(\omega_{\alpha,N})^2 \, \be_\beta \sim \be_{\beta+2\alpha}\, 
\end{align}
and extends straightforward to all other elements of $\verQfull(N)$, because the parabolic Heisenberg algebra commutes with $\omega_{\alpha,N}$.
For each $N\in\nats_0$, this equivalence relation defines a quotient space $\verQ(N)$ from $\verQfull(N)$ by
\begin{align}
\label{eq:erQ}
\verQ(N) \eq \verQfull(N)\big/ \sim\,.
\end{align}

From here on, we always work with the quotient space $\verQ(N)$. The construction of $\verQ(N)$ as a quotient space is particularly fruitful to extend $\verQ(N)$ to a representation of the parabolic algebra $\kfrakpar$. This is subject of the next Proposition.

\begin{proposition}
\label{prop:VerQRepresentation}
For all $N\in\nats_0$, the representation $\verQ(N)$ extends to a representation of the parabolic algebra $\kfrakpar^N$. The element $\Ec_n^\alpha\in\kfrak$ acts on $\be_\gamma$ for $\gamma\in Q$ by
\begin{align}
\label{eq:Action_Ec_on_be_alpha}
\rho^N(\Ec_n^\alpha)\,\be_\gamma 
\aeq \epsilon(\alpha,\gamma) 
\bar S_\alpha^{N,\Qp}\left(-n-(\alpha|\gamma)-1\right)
\be_{\alpha+\gamma}
\end{align}
and the action extends to $\verQ(N)$ by the commutation relations of $\kfrakpar$.
\end{proposition}
\begin{rem}
By this proposition $\verQ(N)$ extends from a representation of $\pheisenberg$ to a representation of the parabolic algebra $\kfrakpar^N$, for which we use the same name $\verQ(N)$. The pull back of $\rho^N$ maps this representation to a representation of $\kfrak$, for which we still use the same symbol $\verQ(N)$ and therefore, we simply write $\Ec_n^\alpha$ for $\rho^N(\Ec_n^\alpha)$, when $\Ec_n^\alpha$ acts on a parabolic representation. 
\end{rem}
\begin{rem}
The state $\be_0\in\verQ(N)$ is a singlet of the maximal compact subalgebra $\kfinite$ of the Lie algebra $\gfinite$ because $\bar S_{\alpha}(-1)=0$ by Corollary \ref{cor:S(-1)}.
\end{rem}
\begin{proof}
The proof contains three parts. The first part shows, that the action \eqref{eq:Action_Ec_on_be_alpha} is well defined on the coset space. The second part shows that \eqref{eq:Action_Ec_on_be_alpha} defines a representation of the parabolic algebra. For example, that $\rho^N(\Ec^\alpha_n)$ can be written in terms of the parabolic generators and $\rho^N$. The third part shows that \eqref{eq:Action_Ec_on_be_alpha} obeys the Lie bracket relation. Then, $\verQ(N)$ extends to a representation of the parabolic algebra.

We show that \eqref{eq:Action_Ec_on_be_alpha} is well defined on the quotient space. By Proposition \ref{prop:WeylGrou} and $\beta \in\Delta$
\begin{align}
\rho^N(\Ec_n^\alpha)(\omega_\beta)^2 \be_{\gamma} \aeq (\omega_\beta)^2\rho^N (\Ec_{n+2(\alpha|\gamma)}^\alpha)\, \be_{\gamma} \eq \epsilon(\alpha,\gamma)\bar S_\alpha^{N,\Qp}(-n-(\alpha|2\beta+\gamma)-1)\be_{\alpha+2\beta+\gamma}\nn\\
\aeq  S_\alpha^{N,\Qp}(-n-(\alpha|2\beta+\gamma))\,\be_{\alpha+2\beta+\gamma} \eq \rho^N(\Ec_n^\alpha)\,\be_{\gamma+2\beta}\,.
\end{align}
Therefore, the action respects the equivalence relation.

Second, we show that \eqref{eq:Action_Ec_on_be_alpha} defines an action of the parabolic algebra $\kfrakpar^N$ consistently. For all $\alpha\in\Delta$ and $0 \leq n \leq 2N+1$ the equation \eqref{eq:Action_Ec_on_be_alpha} defines the action of $\kfrakpar^N$ uniquely. Then, the action of all generators $\Ec_m^\alpha$ for $m\in\ints$ is given in terms of $\rho^N$ and the action of $\kfrakpar^N$. Therefore the action of $\Ec_m^\alpha$ is uniquely given by the action of the elements $\Ec_n^\alpha$ with $0 \leq n\leq 2N+1$ and the ideal relations in Proposition \ref{prop:IdealRelations}. Thus, it remains to show that the ideal relations hold
\begin{align}
\sum_{k=0}^{N+1}&(-1)^k\binom{N+1}{k}\rho^N(\Ec_{m+2k}^\alpha)\,\be_\gamma  \nn\\
\aeq \epsilon(\alpha,\gamma) 
\sum_{k=0}^{N+1}(-1)^k\binom{N+1}{k}\bar S_\alpha^{N,\Qp}\left(-m-2k-(\alpha|\gamma)-1\right)\,
\be_{\alpha+\gamma} \eq 0\,.
\end{align}
The last equation holds because $S_\alpha^{N,\Qp}$ is a polynomial of degree $N$ up to elements in $\ideal_N^{\Uc}$ which act trivial by definition of $\verQ(N)$. This shows that \eqref{eq:Action_Ec_on_be_alpha} defines an action of the parabolic algebra on the states $\be_\beta$ for $\beta \in \Qr$.

Finally, we show that this action of $\kfrakpar$ is indeed a representation and respects the Lie bracket. For all $m,n\in\ints$, $\alpha,\beta\in \Delta$, $\gamma\in \Qr$ and $\Uc(\pheisenberg)$ the action of the generator $\Ec_m^\alpha$ on $U_\Qp\be_\beta$ is uniquely defined by the action of $\Ec_m^\alpha$ on $\be_\gamma$ and the commutation relations of $\Ec^\alpha_m$ with the compact Heisenberg algebra. Therefore, the action of the compact Heisenberg algebra and $\Ec^\alpha_m$ respect the commutation relations. It remains to show that $\Ec^\alpha_m$ and $\Ec^\beta_n$ respect the commutation relations on $\be_\gamma$. 
Inserting Proposition \ref{prop:ExtendedSchurPoly} in \eqref{eq:Action_Ec_on_be_alpha} we find
\begin{align}
\rho^N(\Ec_m^\alpha) \be_\gamma 
\aeq 
\epsilon(\alpha,\gamma)S^{N,\Qp}_{-n-(\alpha|\beta)-1}(\{\Qp^\alpha\})\,\be_{\gamma+\alpha} 
+ 
\epsilon(\alpha,\gamma)S^{N,\Qp}_{n+(\alpha|\beta)-1}(\{-\Qp^\alpha\})\,\be_{\gamma-\alpha}
\nn\\
\aeq 
\epsilon(\alpha,\gamma)S^{\Hc}_{-n-(\alpha|\beta)-1}(\{\Hc^\alpha\})\,\be_{\gamma+\alpha} 
+ 
\epsilon(\alpha,\gamma)S^{\Hc}_{n+(\alpha|\beta)-1}(\{-\Hc^\alpha\})\,\be_{\gamma-\alpha}\, .
\end{align}
But this action is exactly the action on the maximal states of the basic representation and because $\Ec_m^\alpha$ and the compact Heisenberg algebra satisfy the commutation relations, this action necessarily satisfies the commutation relations between $\Ec_m^\alpha$ and $\Ec_n^\beta$. This is illustrated in detail in Appendix \ref{app:Appendix_Details}.
\end{proof}

Now, we are equipped with infinitely many $\kfrak$-representations $\verQ(N)$ and the quotients of $\verQ(N)$ by subrepresentations of $\verQ(N)$. In the theorem below we show that the basic representation can be projected onto these finite dimensional $\kfrak$-representations.

\begin{theorem}
\label{theo:VerQ_In_Basic}
For all $N\in\nats_0$, there exists a $\kfrak$-homomorphism $\Gmap_N:\basic\to\verQ(N)$ given by
\begin{align}
\label{eq:GMap}
\Gmap_N\left(e^{\gamma}\otimes f(\{\alpha\in\Delta\})\right)
\eq 
\Fc_f^{N,\Qp}(\QSet) \be_{\gamma}
\end{align}
which is surjective and commutes with $\kfrak$.
\end{theorem}

\begin{rem}
By this theorem, for all $N\in\nats_0$, the map $\Gmap_N$ projects the basic representation onto $\verQ(N)$ and on quotients of $\verQ(N)$ by subrepresentations of $\verQ(N)$. This is equivalent, to the fact that the kernel of $\Gmap_N$ is an $\kfrak$-invariant subspace 
\begin{align}
W_N = \mathrm{Ker}(\Gmap_N)\subset \basic \,.
\end{align}
\end{rem}
\begin{rem}
The map $\Gmap_N$ together with the action $\Ec_m^\alpha$ on the basic representation, allows to conveniently evaluate the action of $\Ec_m^\alpha$ on the states of $\verQ(N)$ by evaluating $\Ec_m^\alpha$ on the basic representation and then applying $\Gmap_N$. Proposition \ref{prop:Eval_Fc_Q} gives an efficient formula to evaluate $\Gmap_N$ on arbitrary states of the basic representation.
\end{rem}

\begin{proof}
First, we show commutation of $\Gmap_N$ with the compact Heisenberg algebra.
For all $n\in\nats$, $\alpha\in\Delta$, $\gamma \in \Qr$, $f\in\reals[\{\beta\in\Delta\}]$ holds by \eqref{eq:HFf=FHf}
\begin{align}
\Gmap_N\left(\Hc_n^\alpha\, e^\gamma \otimes f \right) 
\aeq  
\Fc_{\Hc_n^\alpha f}^{\Qp}(\QSet)  \be_{\gamma}
\eq 
\Hc_n^\alpha\Fc_{f}^{\Qp}(\QSet) \be_{\gamma} \eq \Hc_n^\alpha\, \Gmap_N\left( e^\beta \otimes f\right) \,.
\end{align}
Therefore, the homomorphism $\Gmap_N$ commutes with the compact Heisenberg algebra. 

Next, we show commutation of $\Gmap_N$ with $\Ec_n^\alpha$ for all $n\in\ints$ on the maximal states. 
Therefore, let us evaluate $\Gmap_N\circ \Ec^\alpha_n$ on a maximal state by using \eqref{eq:Ec_nOnMaximalState} 
\begin{align}
\label{eq:GEcN_on_MaximalStates}
\Gmap_N(\Ec_n^\alpha e^\gamma\otimes 1) 
\aeq 
\Gmap_N\left(\epsilon(\alpha,\gamma)\,e^{\gamma+\alpha} \otimes S_{-n-(\alpha,\gamma)-1}(\alphaSet)+\epsilon(\alpha,\gamma)\,e^{\gamma-\alpha} \otimes S_{n+(\gamma,\alpha)-1}(\{-\alpha\})\right)
\end{align}
By Proposition \ref{prop:kOnMaximalStates} the Schur polynomials $S_n(\{\alpha\})$ are generated by the $\Hc$-Schur polynomials, which allows to replace $S_n(\{\alpha\})$ by $S^{\Hc}_n(\{\Hc^\alpha\})$. Because the compact Heisenberg algebra commutes with $\Gmap_N$, the $\Hc$-Schur polynomials commute with $\Gmap_N$. Then, from \eqref{eq:GEcN_on_MaximalStates} follows 
\begin{align}
\Gmap_N(\Ec_m^\alpha e^\gamma\otimes 1) 
\aeq 
\epsilon(\alpha,\gamma) S^{\Hc}_{-m-(\alpha,\gamma)-1}(\alphaSet)\, \be_{\gamma+\alpha}
\,+\, 
\epsilon(\alpha,\gamma) S^{\Hc}_{m+(\gamma,\alpha)-1}(-\alphaSet)\,\be_{\gamma-\alpha}
\nn\\
\aeq 
\epsilon(\alpha,\gamma) 
\bar S^{N,\Qp}_\alpha({-m-(\alpha,\gamma)-1})\, \be_{\gamma+\alpha} \eq \Ec_m^\alpha\, \be_\gamma
\end{align}
by Proposition \ref{prop:VerQRepresentation}. This proves commutation of $\Ec^\alpha_n$ with $\Gmap_N$ on all maximal states.

It remains to show commutation of $\Ec_m^\alpha$ and $\Gmap_N$ on all states.  
Therefore, we write a state $e^\gamma\otimes f$ for a polynomial $f\in\reals[\{\beta\in\Delta\}]$ as generated from the maximal state by $\Fc_f^{\Hc}$. For $U=\Fc_f^{\Hc}$ we apply the commutation identity \eqref{eq:Commutation_Ecn_Heisenberg}, which provides the appropriate sets of $U_j$ and $n_j$ to commute $\Ec_n^\alpha$ with $\Fc_f^{\Hc}$. It allows us to derive
\begin{align}
\Ec_n^\alpha \Gmap_N\left(e^\gamma\otimes f\right) 
\aeq \Ec_n^\alpha \Gmap_N\left(\Fc_f^{\Hc} e^\gamma\otimes 1\right)
\eq \Ec_n^\alpha \Fc_f^{\Hc}\Gmap_N\left(e^\gamma\otimes 1\right)  
\nn\\
\aeq \sum_{j=1}^{J}U_j \Ec_{n_j}^r \Gmap_N\left( e^\gamma\otimes 1\right)
\eq \sum_{j=1}^{J}U_j \Gmap_N\left(\Ec_{n_j}^\alpha e^\beta\otimes 1\right) \nn\\
\aeq \sum_{j=1}^{J}\Gmap_N\left(U_j \Ec_{n_j}^r e^\gamma\otimes 1\right)
\eq \Gmap_N\left(\Ec_{n_j}^\alpha \Fc_f^{\Hc} e^\gamma\otimes 1\right) 
\eq \Gmap_N\left(\Ec_n^\alpha e^\gamma\otimes f\right)\, , 
\end{align}
which proves commutation of $\kfrak$ and $\Gmap_N$. The homomorphism $\Gmap_N$ is surjective since $\rho^N$ is surjective.
\end{proof}

This theorem shows that for every $N\in\nats$ the surjective homomorphism $\Gmap_N$ projects the basic representation onto $\verQ(N)$ and quotient of $\verQ(N)$ by subrepresentation of $\verQ(N)$. The theorem also implies the existence of infinitely many filtered invariant subspaces as the kernel of $\Gmap_N$, which we give in a corollary after the next theorem.
In the next theorem we show that if and only if the basic representation projects onto a finite dimensional $\kfrak$-representation, which is also a representation of $\kfrakpar$, then the representation is equivalent to $\verQ(N)$ or to the quotient of $\verQ(N)$ by a subrepresentation.

\begin{theorem}
\label{theo:Decomp_Complete}
If for an $N\in\nats_0$ the basic representation projects by $\rho^N$ onto a finite dimensional representation $\Psi$ of $\kfrakpar^N$, then $\Psi$ is 
equivalent to $\verQ(N)$ or to the quotient of $\verQ(N)$ by a subrepresentation of $\verQ(N)$.
\end{theorem}
\begin{rem}
The theorem also implies that a $\kfrak$-invariant subspace in the basic representation for which the quotient of the basic representation by the subspace is a finite dimensional $\kfrakpar$-representation, the subspace is equivalent to $W_N$ or to a subrepresentation of $W_N$ for some $N\in\nats_0$.
\end{rem}

\begin{rem}
\label{rem:k_rep_is_P_rep}
All known finite dimensional $\kfrak$-representations are also representations of the double parabolic algebra $\kfrakpar$. Yet, it is not proven, but we expect that indeed all finite dimensional $\kfrak$-representations are also $\kfrakpar$-representations. Then, for all finite dimensional $\kfrak$-representations which embed in the basic representation, there exists an $N\in\nats_0$ such that the representation is equivalent to $W_N$ or to a subrepresentation of $W_N$. Then, this is the finest $\kfrak$-structure of the basic representation.
\end{rem}

\begin{proof}
If $\Psi$ is the trivial representation, it is equivalent to the quotient of $\verQ(0)$ by $\verQ(0)$. Let us assume $\Psi$ is non-trivial.
Because $\Psi$ is a projection of the basic representation, there exists a surjective homomorphism $\Gmap':\basic \rightarrow \Psi$ which commutes with $\kfrak$.
By Proposition \ref{prop:kOnMaximalStates} and because $\Psi$ is non-trivial by assumption and $\Gmap'$ is surjective, $\Gmap'$ does not vanish on the maximal states. By commutation of $\Gmap'$ and the Heisenberg algebra necessarily $\Psi$ and $\Gmap'$ satisfy
\begin{align}
\Gmap'(e^{\gamma}\otimes f)\eq \Fc_f^{N,\Qp} \Gmap'(e^{\gamma}\otimes f)\, .
\end{align}
Therefore, $\Psi$ as a representation of the parabolic Heisenberg algebra is equivalent to $\verQfull(N)$ or to the quotient of $\verQfull(N)$ with a subrepresentation. 

To show that $\Psi$ is a quotient space of $\verQfull(N)$ w.r.t. the equivalence relation in \eqref{eq:Equi_Relation} and that $\Gmap'=\Gmap_N$, we perform several derivations explained below
\begin{align}
\Gmap'(e^{\gamma+\alpha}\otimes 1) 
\aeq
\epsilon(\alpha,\gamma)\Ec^\alpha_{-(\alpha|\gamma)-1} \Gmap'( e^{\gamma} \otimes 1)
\nn\\ 
\aeq
-\epsilon(\alpha,\gamma)\sum_{m=1}^{N+1}(-1)^m \binom{N+1}{m}\Gmap'\left(\Ec^\alpha_{-(\alpha|\gamma)-1+2m}  e^{\gamma} \otimes 1\right) 
\nn\\
\aeq
-\sum_{m=1}^{N+1}(-1)^m \binom{N+1}{m} S_{2m-2}(\{-\Qp^\alpha\}) \Gmap'( e^{\gamma-\alpha} \otimes 1)  
\nn\\
\aeq
\bar S^{N,-\Qp}_\alpha(-2) \Gmap'( e^{\gamma-\alpha} \otimes 1)  
\eq
\omega^2_\alpha \Gmap'( e^{\gamma-\alpha} \otimes 1)\,.  
\end{align}
In the first line, we used that $\Gmap'$ commutes with  $\Ec^\alpha_{-(\alpha|\gamma)-1}$. In the second line, we used that $\Psi$ has parabolic level $N$, then we applied the ideal relation in Proposition \ref{prop:IdealRelations} and used again commutation of $\Gmap'$ and $\Ec_m^\alpha$. 
For the third line, we evaluate $\Ec_m^\alpha$ on the maximal state with weight $\gamma$, then we used commutation of the $\Gmap'$ with the compact Heisenberg algebra and applied the Lie algebra homomorphism which provides the $\Qp$-Schur polynomials. 
In the fourth line, we applied \eqref{eq:BinomialId} and the fact that $\bar S^{N,\Qp}_\alpha(n)$ are polynomials for the first equation. The last equation derives from Proposition \ref{prop:ExtendedSchurPoly} for $n=2$.
Therefore $\Psi$ is a quotient space w.r.t. the equivalence relation in \eqref{eq:Equi_Relation}. Commutation of $\Gmap'$ and $\Ec_m^\alpha$ implies by Proposition \ref{prop:ExtendedSchurPoly}, that $\Ec_m^\alpha$ acts on $\Psi$ by \eqref{eq:Action_Ec_on_be_alpha}. Thus, $\Psi$ is equivalent to $\verQ(N)$ or to the quotient of $\verQfull(N)$ with a subrepresentation. 
\end{proof}

By these two theorems, we can derive the cosocle filtration of the basic representation, but therefore, it is beneficial to add some notation. For $k\leq N$ let $\verQ(N)_k$ be the projection of $\verQ(N)$ onto the $\kfinite$-representation at parabolic level $k$
\begin{align}
\verQ(N) \eq \bigoplus_{k=0}^N \verQ(N)_k\, .
\end{align}
Then, due to the parabolic grading, for $M\in\nats_0$ also
\begin{align}
\bigoplus_{k=M}^N \verQ(N)_k\,
\end{align}
are $\kfrak$-representations. Using this notation, we provide the cosocle filtration of the basic representation in the next corollary. 
\begin{corollary}
\label{cor:Invariant_Subspace}
For all $N\in\nats_0$, with $\verQ(N)$ and $\Gmap_N:\basic\to\verQ(N)$ as in Proposition \ref{prop:VerQRepresentation} and Theorem \ref{theo:VerQ_In_Basic}, the following identities hold.
\begin{enumerate}[label=\alph*.)]
\item The kernel of $W_N \eq \mathrm{Ker}(\Gmap_N)\subset\basic$ is a non-trivial proper $\kfrak$-invariant subspace and we set $W_{-1}=\basic$.
\item The $\kfrak$-invariant subspaces of $W_N$, for which the quotient of $W_N$ by the subrepresentation is a finite dimensional $\kfrakpar$-representation are $W_{M}$ and subrepresentations of $W_M$ for $M>N$. Equivalently, if $W_N$ projects onto  
a finite dimensional $\kfrakpar$-representations, which is also a $\kfrak$-representation by $\rho$, then there exists a $M\in\nats$, $M\geq N+1$ such that the $\kfrakpar$-representation is equivalent to
\begin{align}
\bigoplus_{k=N+1}^M \verQ(M)_k\, 
\end{align} 
or to a quotient of this representation by a subrepresentation.
\item For $M\in\ints_{\geq -1}$, the quotients of the invariant subspaces
\begin{align}
\frac{W_M}{W_{M+1}} \simeq \verQ(M+1)_{M+1}
\end{align}
are finite-dimensional semisimple $\kfrak$-representations, on which the generators with positive parabolic level act trivial.
\item For $M\in\ints_{\geq -1}$, the cosocle and the radical 
(see Appendix \ref{app:Cosocle_Filtration}) 
of $W_M$ are
\begin{alignat}{6}
&\mathrm{rad}(W_M) &&\eq &&W_{M+1}\, , 
\qquad \mathrm{cosoc}(W_M) &&\simeq &&\verQ(M+1)_{(M+1)}\, .
\end{alignat}
\item The $\kfrak$-cosocle filtration (see Appendix \ref{app:Cosocle_Filtration}) of $\basic$ is $(W_N)_{N\in\nats}$
\begin{align}
\label{eq:CosocleFiltration}
\basic \supset W_0\supset W_1\supset W_2\supset\ldots
\end{align} 
\end{enumerate}
\end{corollary}

This corollary provides the cosocle filtration of the basic representation under $\kfrak$. In the next subsection, we show that the cosocle filtration is indeed an \textit{infinite composition series} (see Appendix \ref{app:Cosocle_Filtration}) of the basic representation.

\subsection{Infinite composition series of basic representation}
\label{subsec:Infinite_Limit}
We prove the quite remarkable fact that the basic representation has an infinite composition series with cosocle filtration under $\kfrak$. That means, we prove that the basic representation has a descending chain of infinitely many invariant subspaces such that the quotient of two adjacent subspaces in the chain is finite dimensional semisimple. As vector spaces, the sum of all these finite dimensional quotients is equivalent to the basic representation. 

First, we derive the infinite completion $\verQ(\infty)$ of the family $(\verQ(N))_N$ and of the cosocle filtration of the basic representation \eqref{eq:CosocleFiltration}. Conversely to $\verQ(N)$, which is a projection of the basic representation, the infinite completion is not a projection of the basic representation but the basic representation embeds in the infinite completion. This will be used to show that the cosocle filtration \eqref{eq:CosocleFiltration} is indeed also an infinite composition series (define in Appendix \ref{app:Cosocle_Filtration}). Clearly, these statements must be properly defined, so first let us define the infinite completion $\verQ(\infty)$ and then show that the basic representation embeds in it. 

For each $N\in\nats_0$, the quotient universal enveloping algebra of the parabolic Heisenberg algebra $\Uc(\pheisenberg)^N$ is defined analogous to $\Uc(\kfrakpar)^N$. The inverse limit of the family $(\Uc(\pheisenberg)^N)_{N\in\nats}$ and $(\verQfull(N))_N$ with the morphism given by the natural projection are called
\begin{align}
\Uc(\pheisenberg)^{\infty}\aeq\lim_{\substack{\longleftarrow \\ N\to \infty}} \Uc(\pheisenberg)^N\, ,\nn\\
\verQfull(\infty) \aeq \lim_{\substack{\longleftarrow \\ N\to \infty}} \verQfull(N) \,\in\, \Uc(\pheisenberg)^\infty \otimes \reals[\Qr] \,.
\end{align}
The map $\omega_{\alpha,\infty}$ on the space $\verQfull(\infty)$ is defined uniquely from $\omega_{\alpha,N}$ and the universal property of the inverse limit 
\begin{align}
\label{eq:WeylGroupElements_infty}
\omega_{\alpha,\infty} \aeq \exp\left(\sum_{k\geq 0}\tfrac{2}{2k+1}\Qp^\alpha_{2k+1,0}\right) \, \in\,  \Uc(\pheisenberg)^\infty \,\,. 
\end{align}
The equivalence relation \eqref{eq:Equi_Relation} extends to an equivalence relation of the completions. For every $\alpha\in\Delta$, and $\be_\gamma^{\infty}\in\verQfull(\infty)$
\begin{align}
(\omega_{\alpha,\infty})^2\be_\gamma^{\infty} \simeq \be_{\gamma+2\alpha}^{\infty} \in \verQfull(\infty)\, ,
\end{align}
which defines the quotient space
\begin{align}
\verQ(\infty) \eq \left(\verQfull(\infty)\big/\sim\right)\,\equiv \lim_{\substack{\longleftarrow \\ N\to \infty}} \verQ(N) \, . 
\end{align}
The action of $\kfrak$ and the generating elements is uniquely given by the universal property of the inverse limit and \eqref{eq:Action_Ec_on_be_alpha} 
\begin{align}
\Fc_f^{\infty,R} = \rho(\Fc_f^{\Hc})\,\in\,\Uc(\pheisenberg)^{\infty}\, .
\end{align}
These elements allow to give the map $\Gmap_{\infty}$ explicitly, which embeds the basic representation in $\verQ(\infty)$. 
By the universal property of the inverse limit, the map
\begin{alignat}{4}
\Gmap_{\infty}\, \colon\; &\quad \basic &&\rightarrow &&\verQ(\infty)\\
 &e^{\gamma}\otimes f(\{\alpha\in\Delta\})\quad 
&&\mapsto
 \quad &&\Fc_f^{\infty,\Qp}(\QSet) \be_{\gamma}\,,
\end{alignat}  
is uniquely defined, such that for every $N\in\nats_0$ the projection of the codomain $\Uc(\pheisenberg)^\infty$ of $\Gmap_{\infty}$ to $\Uc(\pheisenberg)^N$ is the map $\Gmap_N$. 

\begin{proposition}
\label{prop:G_infty_Injective}
The map $\Gmap_{\infty}\colon \basic \rightarrow \verQ(\infty)$ is injective but not surjective. 
\end{proposition}
\begin{rem}
By this proposition the map $\Gmap_N$ and $\Gmap_\infty$ between the modules $\basic$ as a $\kfrak$-module and $\verQ(N)$, $\verQ(\infty)$ as a $\kfrakpar$-module behave similar to the Lie algebra homomorphism $\rho^N$ and $\rho$ from $\kfrak$ to $\kfrakpar^N$ resp. $\kfrakpar$. 
By Proposition 9 and 11 in  \cite{Kleinschmidt:2021agj}, the map $\rho^N$ is surjective but not injective, while its completion $\rho$ is injective but not surjective. In particular the elements in $\kfrakpar$ which contain only finitely many terms in the formal power series are not in the image of $\rho$. The basic representation does not project onto the infinite completion $\verQ(\infty)$, but rather embeds in it.
\end{rem}

\begin{proof}
First, we show that $\Gmap_{\infty}$ is injective. 
Because the map $\rho$ is injective, the generating element, as a function $\Fc^{\infty,\Qp}_f: f\in\reals[\{\alpha\in\Delta\}]\mapsto\Uc(\pheisenberg)$ is injective. Therefore, for $\gamma\in Q$, the map $\Gmap_\infty$ is injective on the states $\{e^\gamma\otimes f\forwhich f\in\reals[\{\alpha\in\Delta\}]\}$. 

The elements $\be_{\gamma}$ and $\be_{\gamma'}$ are independent if $\gamma-\gamma'\notin 2\ints\Delta$ and therefore it remains to show that the images of $\Gmap_\infty$ on the states 
\begin{align}
\label{eq:help_Prop_G_infty_3}
\{e^{\gamma+2l\alpha}\otimes f\forwhich f\in\reals[\{\beta\in\Delta\},\,\alpha\in\Delta,\, l\in\nats]\}
\end{align}
are linearly independent. The image of $\Gmap_\infty$ on $e^{\gamma+2l\alpha}\otimes f$ and on $e^{\gamma+2k\alpha}\otimes f$ differ by a factor of $(\omega_{\alpha,\infty})^{l-k}$. Therefore, it remains to show that the images of $\Gmap_{\infty}$ on the states 
\begin{align}
\label{eq:help_Prop_G_infty_4}
\{e^{\gamma}\otimes f\forwhich f\in\reals[\{\beta\in\Delta\}]\}
\end{align}
are linearly independent to the images of $\Gmap_{\infty}$ on the states
\begin{align}
\label{eq:help_Prop_G_infty_5}
\{e^{\gamma+2l\alpha}\otimes f\forwhich f\in\reals[\{\beta\in\Delta\},\,\alpha\in\Delta]\}
\end{align}
for $l\in\nats$. The remaining independents relations of the image of the states in \eqref{eq:help_Prop_G_infty_3} follow by division with $\omega_{\alpha,\infty}$. 
To show that the image on the states in \eqref{eq:help_Prop_G_infty_4} and \eqref{eq:help_Prop_G_infty_5} are independent we show that the generating elements $\Fc_{f}^{\infty,\Qp}$ are linearly independent from the elements $\Fc_{f'}^{\infty,\Qp}(\omega_{\alpha,\infty})^{2l}$ for $l\neq 0$. By the series expansion \eqref{eq:WeylGroupElements}, it is sufficient to show that there is no elements in $\cheisenberg$ which is mapped to $\sum_{k\geq 0}\tfrac{4l}{2k+1}\Qp_{2k+1,0}^\alpha$ by $\rho$. It is clear, that this is true for all elements $\Hc^\beta_m\in\kfrak$ with $\beta\neq \pm \alpha$ or $m$ odd. Suppose now
\begin{align}
0\aeq \rho(\sum_{i=1}^Iv_i\Hc^{\alpha}_{2n_i}) + v_0\sum_{k\geq 0}\tfrac{4l}{2k+1}\Qp_{2k+1,0}^\alpha\nn\\
\aeq \sum_{i=1}^Iv_i \sum_{k\geq 0} \sum_{s=1}^{2n_i-1}c^{(2n_i)}_s (2k+1)^s \Qp_{2k+1,0}^\alpha
\,+\, 
+ v_0\sum_{k\geq 0}\tfrac{4l}{2k+1}\Qp_{2k+1,0}^\alpha .
\end{align}
where $c^{(n)}_s\in\reals$ are the polynomial coefficients. While $\tfrac{1}{2k+1}$ has no finite series expansion, the first term in the equation are polynomials in $2k+1$ of finite degree. But since this equation must hold for all $k\in\nats_0$, the second and first term must cancel independently.

We show that $\Gmap_{\infty}$ is not surjective. By Proposition 9 in \cite{Kleinschmidt:2021agj}, the element $\Qp_{1,0}^\alpha$ is not in the image of $\rho$. Therefore, there exists no polynomial $f\in\reals[\{\beta\in\Delta\}]$ and $\gamma\in Q$ such that $\Gmap_\infty(e^{\gamma}\otimes f) \eq \Qp_{1,0}^\alpha\be_0$. 
\end{proof}

\begin{corollary}
\label{prop:Infinite_Composition_series}
The cosocle filtration $(W_N)_{N\in\ints_{\geq -1}}$ in Corollary \ref{cor:Invariant_Subspace} is an infinite composition series of the basic representation under the action of $\kfrak$
\begin{align}
\label{eq:Inverse_Limit_WN}
\lim_{\substack{\longleftarrow \\ N\to \infty}} W_N \eq \{0\}\, .
\end{align}
\end{corollary}

\begin{proof}
By the definition of an infinite composition series in Appendix \ref{app:Cosocle_Filtration} it remains to show \eqref{eq:Inverse_Limit_WN}. For $\verQ(-1)=\{0\}$, the family $(\verQ(N)_{N\in\ints_{\geq -1}}$
together with the natural projection $\verQ(N+1)\to \verQ(N)$ provides the family $(W_N=\mathrm{Ker}(\Gmap_{N}))_{N\in\ints_{\geq -1}}$ with $\basic/W_N \simeq \verQ(N)$ and implies the natural projective embedding $W_{N+1}\to W_N\supset W_{N+1}$. Then, by the universal property of the inverse limit applied to $\Gmap_N$ which provides $\Gmap_\infty$, the inverse limit of the family $(W_N=\mathrm{Ker}(\Gmap_{N}))_N$ is $W_\infty = \mathrm{Ker}(\Gmap_{\infty}) = \{0\}$ by Proposition \ref{prop:G_infty_Injective}. 
\end{proof}

In this section we derived infinitely many $\kfrak$-subrepresentations of the basic representation as the kernels of infinitely many different $\kfrak$-projections of the basic representations on finite dimensional $\kfrak$-representations. We proved, that these are all finite dimensional $\kfrakpar$-representations onto which the basic representation can project by the pull back of $\rho$. This allowed us to provide the cosocle filtration and to show that it is an infinite composition series of the basic representation. In particular, as vector spaces or $\kfinite$-representations, the basic representation is isomorphic to a direct sum of infinitely many finite-dimensional semisimple $\kfrak$-representations.

In the next section we provide examples and applications of the results.

\section{Examples and applications to supergravity}
\label{sec:Applications_and_examples}
In this section we provide some explicit examples of the projection of the basic representation onto finite dimensional $\kfrak$-representations and we give applications of our results to supergravity in two dimensions. 

First, in Subsection \ref{subsec:Embedding_Map_Explicit} we prove a Proposition which allows to evaluate the generating elements $\Fc^\Hc$ to any state of the basic representation in a closed form and which simplifies the evaluation of $\Fc^{\Qp}$ significantly. Then, we add some comments on the finite dimensional compact subalgebra $\kfinite$ to further analyze the $\kfrak$-representations as $\kfinite$-representations.

Second, in Subsection \ref{subsec:General_examples_embedding} the results of the previous section are applied to obtain explicit examples of invariant subspaces and about the projection of the basic representation onto $\kfrakpar$-representations of parabolic level zero and parabolic level one. We use the proven proposition to give the projection homomorphism explicitly.

Third, in Subsection \ref{subsec:e9_Basic_decomposition} we apply the results in the previous section to the basic representation of the split real $\enn$ algebra with maximal compact subalgebra $\kenn$. We find the $\kenn$-representations of parabolic level zero, one and two which are projections of the basic representation and compare these representations with tensor products of representations of the single parabolic algebra constructed in \cite{Kleinschmidt:2021agj}. This provides an  $\so(16)$-covariant formulation of the $\kenn$-representations on which the basic representation projects. At the end, we discuss applications to supergravity in $D=2$ dimensions.

\subsection{Evaluation of projection} 
\label{subsec:Embedding_Map_Explicit}
To understand the $\kfrak$-subrepresentations of the basic representations explicitly, the homomorphism $\Gmap_N$ must be evaluated on the states $e^\gamma \otimes f$ of the basic representation \eqref{eq:GMap}. By Theorem \ref{theo:VerQ_In_Basic} it remains to evaluate the generating elements $\Fc^{\Qp}_f = \rho(\Fc^{\Hc}_f)$ for polynomials $f\in\reals[\{\alpha\in\Delta\}]$. 
The next Proposition provides an explicit expressions of the generating elements $\Fc^{\Hc}_f$.

\begin{proposition}
\label{prop:Eval_Fc_Q}
Let $\{v^i \in \mathring{\mathfrak{h}}^{\vee}\forwhich 1\leq i\leq r\}$ be an orthonormal basis of $\mathring{\mathfrak{h}}^{\vee}$ and let $\{H^{v,i}\in\mathring{\mathfrak{h}}\forwhich 1\leq i\leq r\}$ be the dual basis with $\Hc^{v,i}_n = \half H^{v,i}(t^n-t^{-n})$, then 
\begin{alignat}{3}
\label{eq:Eval_GeneratingFcH}
a.)& \qquad \qquad && \Fc^{\Hc}_{f_1(\{v^i\})f_2(\{v^j\})}\eq \Fc^{\Hc}_{f_1(\{v^i\}}\Fc^{\Hc}_{f(\{v^j\})} \qquad\quad \textrm{for }i\neq j,&&\nn\\ 
b.)& && \Fc^{\Hc}_{(v_{-l}^i)^n}
\eq
\sum_{k=0}^n c_{l,n,k} \left(\Hc^{v,i}_l\right)^k e^\gamma\otimes 1 
\nn\\
& && \quad c_{l,n,k} \eq 
\begin{cases}
(-1)^n\frac{(2l(n-k-1))!_{(4l)}(2n)!_{(2)}}{k!(2n-2k)!_{(2)}} & \text{ for }n-k \text{ even}\\
0& \text{ otherwise}
\end{cases}\, .&& 
\end{alignat}
where for $n,N\in\nats$, $N!_{(n)} = N(N-n)!_{(n)}$ for $N\geq n$, $N!_{(n)} = N$ for $1\leq n\leq N$, $0!_{(n)} = (-N)!_{(n)} = 1$.
\end{proposition}
\begin{rem}
By \eqref{eq:Eval_GeneratingFcH} it is straightforward to evaluate the generating elements $\Fc^{\Hc}_f$ for any $f\in\reals[\{\alpha\in\Delta\}]$. The Lie algebra homomorphism $\rho^N$ maps group like the generating elements $\Fc^{\Hc}$ to the parabolic generating elements $\Fc^\Qp\in\Uc(\pheisenberg)$. Because of the parabolic grading, and because every element of the parabolic Heisenberg algebra has at least parabolic level one, the Lie algebra homomorphism $\rho^N$ is easy evaluated for low $N$. 
\end{rem}

\begin{proof}
The loop generators of the dual roots $H^{v,i}_n$ commute and therefore $H^{v,i}_n$ does not act on $e^{\gamma}\otimes (v^j_{-l})^n$ for $i\neq j$ and $l,n\in\nats$, which proves \textit{a}.

We prove \textit{b} by induction in $n$. For $n=0$ the statement is  true and for $n=1$, $c_{l,1,k} = -2\delta_{k,1}$ and
\begin{align}
\Fc_{v^i_{-l}} = -2\Hc_{l}^{v,i}\,,
\end{align}
the statement is also true. Assuming \textit{b} is true for $n\leq N$ and applying the commutation $[H^{v,i}_{-l},(\Hc^{v,i}_{-l})^k] = -kl(\Hc^{v,i}_{-l})^{k-1}$ results in
\begin{align}
\label{eq:prf_prop_Form_FH}
\Fc_{(v^i_{-l})^{N+1}} 
\aeq H^{v,i}_{-l}\sum_{k=0}^n c_{l,N,k} (\Hc_{-l}^{v,i})^k\nn\\
\aeq -lc_{l,n,1} +
\sum_{k=1}^N -\left((k+1)lc_{l,N,k+1}+2c_{l,N,k-1}\right) (\Hc_{-l}^{v,i})^k 
- 2c_{l,N,N}(\Hc_{-l}^{v,i})^{N+1}\,.
\end{align}
From the explicit form of $c$ follow the identities
\begin{align}
c_{l,N+1,0} \aeq -lc_{l,N,1}\, ,\qquad c_{l,N+1,N+1} \eq -2c_{l,N,1}\nn\\
c_{l,N+1,k} \aeq -l(k+1)c_{l,N,k+1}) - 2c_{l,n,k-1}\, ,
\end{align}
which, inserted in \eqref{eq:prf_prop_Form_FH} prove \textit{b}.
\end{proof}

The Proposition allows to evaluate the homomorphism $\Gmap_N$, which projects $\basic$ onto $\verQ(N)$. To analyze the $\kfrak$-representations $\verQ(N)$, they may be decomposed into $\kfinite$-representations. Therefore, let us first consider the quotient algebra $\kfrakpar^0$ which is also a finite dimensional subalgebra of the double parabolic algebra
\begin{align}
\kfrakpar^0 \eq \left\lbrace P_{0,a}^\alpha\forwhich \alpha\in\Delta,\; a\in\{0,1\}\right\rbrace\,.
\end{align}
By definition these are the only elements of the double parabolic algebra which preserve the parabolic level of states in the parabolic representation. The Lie bracket of these elements is
\begin{align}
\left[P^\alpha_{0,a},P_{0,b}^\beta\right] \aeq 
\epsilon(\alpha,\beta) \,\left(P_{0,a+b}^{\alpha+\beta}+ P^{\alpha-\beta}_{0,a+b}\right)\, .
\end{align}
Therefore, $\kfrakpar^0$ is a $\ints_2$ graded algebra $\kfrakpar^0 \eq \kfinite_0\oplus\kfinite_1$, where $\kfinite_0\simeq\kfinite$ as Lie algebras and $\kfinite_1\simeq\kfinite$ as vector spaces. In particular $\kfinite_0$ is a subalgebra as well and equivalent to the semisimple Lie algebra $\kfinite$. This analysis helps to analysis the decomposition of $\verQ(N)$ in $\kfinite$-representations in Subsection \ref{subsec:e9_Basic_decomposition}. However, first, let us provide general examples of the projections $\Gmap_N$ in the next subsection.

\subsection{General examples}
\label{subsec:General_examples_embedding}
In this section we provide the projection of the basic representation on the parabolic level zero and parabolic level one representations $\verQ(0)$ and $\verQ(1)$ and give the projections $\Gmap_0$ and $\Gmap_1$ explicitly. We analyze the representations $\verQ(0)$ and $\verQ(1)$ in more detail and add some comments on the decomposition into $\kfinite$-representations. 
It is straightforward to generalize this to any parabolic level $N\in\nats_0$ and representation $\verQ(N)$.

\subsubsection{Parabolic level zero}
\label{subsubsec:Par_0}
The simplest $\kfrak$-representation onto which the basic representation can be projected is of parabolic level zero. By Theorem \ref{theo:VerQ_In_Basic} this representation is $\verQ(0)$ given in Proposition \ref{prop:VerQRepresentation}. Let us analyze this representation in more detail.

Because the representation has parabolic level zero, all parabolic generators of higher parabolic level than zero act trivial. In particular $\omega^2_\alpha -1 \in\ideal_0$ acts trivial and implies that for $\alpha\in\Delta$ and $\gamma\in \Qr$, $\be_{\gamma+2\alpha } - \be_{\gamma}\sim 0 \in\verQ(0)$. Therefore, the representation has dimension $\textrm{dim}(\verQ(0)) = 2^r$, where $r$ is the rank of the finite algebra $\gfinite$.
The action of the parabolic generators of level $0$ is given in Proposition \ref{prop:VerQRepresentation} in terms of the extended $\Qp$-Schur polynomials 
$S^{0,\Qp}({2n}) = 1$ and $S^{0,\Qp}({2n+1})=0$ \eqref{eq:App_Q_Schur_Poly} 
\begin{align}
P^\alpha_{0,a} \be_{\gamma} \eq 
\epsilon(\alpha,\gamma)
\begin{cases}
0&\textrm{if }(\alpha|\gamma) + a \textrm{ even}\\
\be_{\gamma+\alpha}&\textrm{if }(\alpha|\gamma) + a \textrm{ odd.}
\end{cases}
\end{align}
This action 
allows us to analyze the decomposition of $\verQ(0)$ into $\kfinite$-representations. 

The representation is generated from the action of $P_{0,1}^\alpha$ and $P_{0,1}^{(\alpha} P_{0,1}^{\beta)}$ for different $\alpha\in\Delta$ on $\be_0$, because for $\alpha,\beta\in\Delta$, and $\alpha+2\beta\in\Delta$, then $\beta = -\alpha$.
Explicitly we obtain 
\begin{alignat}{23}
\label{eq:GeneralPar0}
\be_{0} \,&\ \longrightarrow\ \, &&\mathbf{1}_{\kfinite}
\nn\\
P_{0,1}^\alpha \be_0 \eq \be_\alpha \, & \ \longrightarrow\ \, &&\adj_{\kfinite}
\nn\\
P_{0,1}^\alpha P_{0,1}^\beta \be_0 \eq\half \{P_{0,1}^\alpha ,P_{0,1}^\beta\} \be_0 \, &\  \ \; \subset  \, &&\left(\adj_{\kfinite}\otimes_{\textrm{sym}} \adj_{\kfinite}\right)\,.
\end{alignat}
The first line is a singlet under $\kfinite$. The second line is the adjoint representation of $\kfinite$ and the last line takes values in the symmetric product of two adjoint representations of $\kfinite$.
The $\kfinite$ singlet in the last line is proportional to $\be_0$ because $P_{0,1}^\beta P_{0,1}^\alpha = \eta^{\alpha,\beta} \be_0$. All together, $\verQ(0)\subset \adj(\kfinite)\otimes \adj(\kfinite)$.
This provides some inside in the decomposition of $\verQ(0)$ into $\kfinite$-representations. 

The basic representation is projected onto $\verQ(0)$ with the homomorphism $\Gmap_0$ in Theorem \ref{theo:VerQ_In_Basic}. By Proposition \ref{prop:Eval_Fc_Q} and the Lie algebra homomorphism $\rho^0$ we can evaluate the projection in terms of the generating elements $\Fc^{0,\Qp}_f$ efficiently. For the case of parabolic level zero, all Heisenberg generators act trivial and one evaluates \eqref{eq:Eval_GeneratingFcH} for $k=0$. Then the projection $\Gmap_0$ acts on the basic representation by
\begin{align}
\Gmap_0\left(e^{\gamma}\otimes \prod_{i=1}^r (v^{i}_{-l_i})^{n_i}\right)
\aeq
\begin{cases}
\left(\prod_{i=1}^r(2l_i(n_i-1))!_{4l_i}\right)\be_{\gamma} & \textrm{for all }n_i\textrm{ even}\\
0& \textrm{otherwise}\,.
\end{cases}
\end{align}
By Corollary \ref{cor:Invariant_Subspace}, the kernel of $\Gmap_0$ is a $\kfrak$-invariant subspace. For instance all elements of the basic representation which contain in vertex operator realization an odd power of a indexed simple root are in the kernel of $\Gmap_0$.

\subsubsection{Parabolic level one}
\label{subsubsec:Par_1}
Let us decompose the parabolic level one representations $\verQ(1)$. The states of the representation are 
\begin{align}
\verQ(1) = \left\lbrace\be_{\gamma}\, ,\ \Qp_{1,a}^{\alpha^i}\be_{\gamma}\forwhich \gamma\in Q ,\, \alpha^i\in\Pi,\, a\in\{0,1\}\right\rbrace\, .
\end{align}
Therefore, the representation has dimension $\textrm{dim}(\verQ(1)) = 2^r + 2\cdot r\cdot 2^r$. The states $\Qp_{1,a}^{\alpha^i}\be_\gamma$ are strictly of parabolic level one while the states $\be_\gamma$ have contributions from parabolic level zero and one. 
Thus, to parabolic level zero, the representation $\verQ(1)$ has $2^r$ number of states contained in $\be_\gamma$. The remaining states are $\Qp^i_{1,a}\be_\gamma$, which accounts for $2\cdot r\cdot 2^r$ number of states at parabolic level one. These are all the states of the representation and therefore, parabolic level one contribution in $\be_\gamma$ are necessarily linear combinations of the $2\cdot r\cdot 2^r$ states $\Qp^{\alpha^i}_{1,a}\be_\gamma$.

The generators of parabolic level higher than one act trivial and the action of the other generators is given in \eqref{eq:Action_Ec_on_be_alpha} and in terms of the extended $\Qp$-Schur polynomials
\begin{align}
S_\alpha^{1,\Qp}(2n) \eq
1 + 4n \Qp^\alpha_{1,0},\qquad
S_\alpha^{1,\Qp}(2n+1) \eq 4(n+1)\Qp^\alpha_{1,1}\, .
\end{align}
The action of the parabolic generators is then evaluated from \eqref{eq:Action_Ec_on_be_alpha}
\begin{align}
P^\alpha_{0,a} \be_{\gamma} 
\aeq 
\epsilon(\alpha,\gamma)
\begin{cases}
1-2\left((\alpha|\gamma)+2\right)\Qp^\alpha_{1,0}
& \textrm{ if }(\alpha|\gamma)+a \textrm{ is odd}\\
-2(\alpha|\gamma)\Qp^\alpha_{1,1}
&\textrm{ if }(\alpha|\gamma)+a  \textrm{ is even.}
\end{cases}
\nn\\
P_{1,a}^\alpha\be_\gamma 
\aeq 
\epsilon(\alpha,\beta)\Qp^\alpha_{1,a+1+(\alpha,\beta)}\be_{\alpha+\gamma}\,.
\end{align}
On the states at level one, the generators $P^\alpha_{1,a}$ act trivial while the action of $P^\alpha_{0,a}$ is given by the commutation relations. 

The homomorphism $\Gmap_1$ from Theorem \ref{theo:VerQ_In_Basic} 
can be evaluated by Proposition \ref{prop:Eval_Fc_Q}. For parabolic level one, products of more than one Heisenberg generators act trivial. Therefore, \eqref{eq:Action_Ec_on_be_alpha} must be evaluated at $k=0$ and $k=1$ and thus, the homomorphism $\Gmap_1$ of $\verQ(1)$ is 
\begin{align}
&\Gmap_1\left(e^{\gamma}\otimes \prod_{i=1}^r (v^{i}_{-l_i})^{n_i}\right)
\nn\\
&\ \eq
\begin{cases}
\left(\prod_{i=1}^r(2l_i(n_i-1))!_{4l_i}\right)\be_{\gamma} & \textrm{for all }n_i\textrm{ even}\\
\left(\prod_{j\neq i=1}^r(2l_i(n_i-1))!_{4j_i}\right)4n_j (2l_j(n-2))!_{(4l_j)}\Qp^{\alpha^j}_{1,l_j} \be_{\gamma} & \textrm{for }n_j \textrm{ odd, } n_i \textrm{ even for }i\neq j\\
0& \textrm{otherwise}\,.
\end{cases}
\end{align}
It is straightforward to evaluate the projection of the basic representation onto even higher parabolic representations by applying Proposition \ref{prop:VerQRepresentation}, Theorem \ref{theo:VerQ_In_Basic} and Proposition \ref{prop:Eval_Fc_Q}.

To obtain a better understanding of the $\kfrak$-representations in Proposition \ref{prop:VerQRepresentation} on which the basic representation projects, we decompose these representations into $\kfinite$-representations.
In general, this does not lead to new obstacles but may be achieved with standard techniques of representation theory, but it requires a suitable basis change of the generators in \eqref{eq:Com_Finite_Algebra} to a $\kfinite$-covariant basis. Therefore it seems useful to restrict to a specific affine Kac-Moody algebra which we do in the next subsection.

\subsection{Projections of $\mathfrak{e}_9$ basic representation}
\label{subsec:e9_Basic_decomposition}
Certainly, one of the most interesting split real simply-laced affine Kac-Moody algebras is $\enn$ with maximal compact subalgebra $\kenn$. In this section we decompose the finite dimensional $\kenn$-representations $\verQ(0)$, $\verQ(1)$ and $\verQ(2)$ on which the basic representation projects under the finite compact subalgebra $\mathring{k}(\enn) = \so(16)$. This is an $\so(16)$-covariant form of the $\kenn$-representations which has the additional benefit, that it is significantly easier to obtain subrepresentations of $\verQ(N)$.

Instead of performing the explicit $\so(16)$-decomposition of $\verQ(N)$, we rather apply the construction of $\so(16)$-covariant $\kenn$-representations in \cite{Kleinschmidt:2021agj} to build up $\so(16)$-covariant $\kenn$-representations which are equivalent to $\verQ(0)$, $\verQ(1)$ and $\verQ(2)$. This is equivalent to a direct decomposition of these $\kenn$-representations. 
In this work we restrict the analysis to dimensional arguments, because the explicit analysis to show that the $\so(16)$-covariant $\kenn$-representations are equivalent to $\verQ(N)$ for some $N\in\nats_0$ would exceed the scope of this paper. However, the explicit analysis finds important application in supergravity in two dimensions as it describes the yet unknown supersymmetry of vector fields and the decomposition of the embedding tensor into fermion bilinears. Therefore, we present it in a subsequent paper.  

First, we introduce an $\so(16)$-covariant notation of the algebras of interest. Second, we construct $\so(16)$-covariant $\kenn$-representation along \cite{Kleinschmidt:2021agj} which are equivalent to $\verQ(N)$ for $N\leq 2$. Third, we apply these concrete results to supergravity in two dimensions.

\subsubsection{$\so(16)$-covariant formulation}
\label{subsubsec:e9_so16_covariant}
A thorough introduction of $\kenn$ and the single parabolic algebra in $\so(16)$-covariant formulation is provided in \cite{Kleinschmidt:2021agj} section 5.1 and 5.2. Here, only the necessary notation and identities for this work are introduced.
 
The finite dimensional Lie algebra $\eee$ has dimension $\dim(\eee)=248$, its maximal subalgebra is $\so(16)$ with $120$ generators $X^{IJ} = -X^{JI}$ for $1\leq I,J\leq 16$. The non-compact orthogonal complement transforms as the $128_{\so(16)}$ dimensional $\so(16)$-spinor representation with basis elements $Y^A$ for $1\leq A\leq 128$. In this $\so(16)$-covariant formulation, the Lie bracket and the Killing form are
\begin{multicols}{2}
\noindent
\begin{align}
\left[ X^{IJ} , X^{KL} \right] &= 4 \delta^{JK} X^{IL} \,,\nn\\
\left[ X^{IJ} , Y^B \right] &= -\half \Gamma^{IJ}_{AB} Y^B\,,\nn\\
\left[ Y^A, Y^B \right] &= \tfrac{1}{4} \Gamma^{IJ}_{AB} X^{IJ}\,,\nn 
\end{align}
\begin{align}
\label{eq:LieBracket_ee}
\kappa(X^{IJ},X^{KL}) \aeq -2\delta^{IK}\delta^{JL}\, ,\nn\\
\kappa(Y^{A},Y^{B}) \aeq \delta^{AB}\, . 
\end{align}
\end{multicols}
For $\eee$ with positive roots $\alpha\in\Delta_+$, we formally write the basic change from $E^\alpha$ and $H^\alpha$ to the $\so(16)$-covariant basis by
\begin{align}
X^{IJ} \aeq \sum_{\alpha\in\Delta_+}x_{\alpha}^{IJ} \left(E^\alpha + E^{-\alpha}\right)\nn\\
Y^{A} \aeq \sum_{\alpha\in\Delta_+}y_{\alpha}^{A} \left(E^\alpha - E^{-\alpha}\right) \, +\, \sum_{i=1}^Ry_{i}^{A} H^{\alpha^i} 
\end{align}
For the inverse basis change one might use the coefficients $x^\alpha_{IJ}$ and $y^\alpha_{A}$, $y^i_{A}$. To explicitly branch the finite dimensional $\kenn$-representations $\verQ(N)$ under $\so(16)$ many useful relations of the coefficients of the basis change may be derived. This explicit analysis is performed in a subsequent paper while we restrict here on dimensional arguments.
 
The Lie bracket of the affine algebra $\enn$ is given by \eqref{eq:LieBracket_ee} and \eqref{eq:Lie_Bracket_Affine}. 
Therefore, this basis change extends straightforward to the affine algebra $\enn$ and the maximal compact subalgebra $\kenn$. 
On the parabolic algebra $\kfrakpar(\enn)$ the basis change is
\begin{alignat}{2}
\label{eq:Basis_Change_Para}
A^{IJ}_{2k,a} \aeq X^{IJ}\otimes u^{2k}\otimes r^a &&\eq \sum_{\alpha\in\Delta_+} x_{\alpha}^{IJ} P^\alpha_{2k,a}\nn\\
S^{A}_{2k+1,a} \aeq Y^{A}\otimes u^{2k+1}\otimes r^a &&\eq \sum_{\alpha\in\Delta_+}y_{\alpha}^{A} P^\alpha_{2k+1,a} \, +\, \sum_{i=1}^r y_{i}^{A} \Qp^{\alpha^i}_{2k+1,a} \, .
\end{alignat}
Setting $r=1$ the double parabolic algebra restricts to the  single parabolic algebra of $\enn$ in \cite{Kleinschmidt:2021agj} section 5. 

This sets up the $\so(16)$-covariant notation of $\enn$, $\kenn$ and the parabolic algebra $\kfrakpar(\enn)$. Next, let us also give the $\eee$ and $\so(16)$-decomposition of the first few loop levels of the basic representation
\begin{align}
\label{eq:Basic_e9_Decomposed}
\basic \aeq \left(\mathbf{1}_{\eee}\right)_0\; \oplus\; \left( \mathbf{248}_{\eee}\right)_1\; \oplus\;\left(\mathbf{1}_{\eee}\;\oplus\; \mathbf{248}_{\eee} \oplus \;\mathbf{3875}_{\eee}\right)_2\; \oplus\; \ldots\nn\\[5pt]
\aeq \left( \mathbf{1}_{\so(16)}\right)_0 \;\oplus\; 
\left(\mathbf{120}_{\so(16)}\;\oplus\; \mathbf{128}_{\so(16)}\right)_1  \; \oplus\; \\[5pt]
& \;  \left(\mathbf{1}_{\so(16)}\;\oplus
\mathbf{120}_{\so(16)}\,
\oplus\,\mathbf{135}_{\so(16)}\,
\oplus\,\mathbf{1820}_{\so(16)}\,
\oplus\,\mathbf{128}_{\so(16)}\,
\oplus \,\mathbf{1920}_{\so(16)}\,\right)_2\, \oplus\ldots \nn
\end{align}
The decomposition to higher levels is in \eqref{app:Decompose_enn_Basic}, it provides some intuition about the $\so(16)$-representations to expect in the $\so(16)$-branching of $\verQ(N)$ on which the basic representation projects $\kfrak$-covariant. For example, the basic representation does not contain at any level the $\mathbf{16}_{\so(16)}$ vector representation nor the conjugate spinor representation $\overline{\mathbf{128}}_{\so(16)}$.
More generally, the basic representation only contains $\so(16)$-representations which are also in the root lattice of $\eee$. 
We use these constraints about possible $\so(16)$-representations in the basic representation to construct the $\so(16)$-decomposition of $\verQ(N)$ for $N\leq 2$ in the next subsection

\subsubsection{Embedding of $\so(16)$-covariant representations and applications}
\label{subsubsec:Embedding_so(16)_covariant}
The plan of this subsection is to build up $\so(16)$-covariant $\kenn$-representations along section 5.3 in \cite{Kleinschmidt:2021agj} which have the same dimensions as $\verQ(N)$ for $N=0,1,2$ at every parabolic level. 
This is a strong indication that the different $\so(16)$-covariant representations are indeed equivalent to the different $\verQ(N)$. 
Therefore, first we argue that the basic representation cannot project on a pure representation of the single parabolic algebra but only on a tensor product. Then, we sketch the construction of representations of the single parabolic algebra in \cite{Kleinschmidt:2021agj} and we build up the representations $\verQ(N)$ for $N=0,1,2$ from tensor products of representations of the single parabolic algebra. 

By the different action of $P_{0,0}^\alpha$ and $P_{0,1}^\alpha$ of the double parabolic generators on $\verQ(N)$ in \eqref{eq:GeneralPar0} it is evident that the basic representation cannot project onto pure representation of the single parabolic algebra. Therefore, the representations onto which the basic representation can be projected must be at least a tensor product of two representations of the single parabolic algebra which are pulled back to representations of $\kfrak$ by the two different Lie algebra homomorphisms $\rho_{\pm}$ (for details see also Footnote \ref{ftn:Singel_Double_Parabolic}). 
Henceforth, we consider the tensor product of two representations $\Phi$ and $\Psi$ of the single parabolic algebra and it remains to build up $\Phi$ and $\Psi$ such that the tensor product is equivalent to $\verQ(0)$, $\verQ(1)$ and $\verQ(2)$. 
Because the representations decompose in the different parabolic levels, the analysis can be carried out for the different levels individually. Therefore, we branch $\Phi$ and $\Psi$ in $\so(16)$-representations $\varphi^X_k$ and $\psi^X_k$ where $X$ is the index of the $\so(16)$ representation and $k$ is the parabolic level. 

Now, let us describe how to construct the representations $\Phi$ and $\Psi$ of the single parabolic algebra. One start with an $\so(16)$-representations $\varphi^X_0$ and $\psi^Y_0$ at parabolic level zero. The next parabolic levels of the representation are constructed as a Verma module of the parabolic generators  with positive parabolic level \eqref{eq:Basis_Change_Para} acting on $\varphi^X_0$ and $\psi^Y_0$.  From the Verma module, finite dimensional representations are obtained by quotienting with subrepresentations. Here, an $\so(16)$-covariant basis of the generators \eqref{eq:Basis_Change_Para} already implies that the full representation will be $\so(16)$-covariant.

\paragraph{Parabolic level zero.}
For $N\in\nats$, the elements of $\verQ(N)$ at parabolic level zero are
\begin{align}
\left\lbrace \be_\gamma \in\verQ(0)\forwhich \gamma\in \Qr\right\rbrace\, .
\end{align}
These are $2^8 = 256$ inequivalent elements by the equivalence relation \eqref{eq:Equi_Relation}.
Thus, the tensor product of the two representation $\varphi_0$ and $\psi_0$ must have dimension $256$ and the decomposition of the tensor product into $\so(16)$-representations takes weights in the $\eee$ weight lattice.
The general pattern of $\verQ(0)$ in \eqref{eq:GeneralPar0} also predicts the tensor product to decompose into an $\so(16)$-singlet, an $\so(16)$-adjoint representation and an $\so(16)$-representation in the symmetric tensor product of two $\so(16)$-adjoint representations.

These constraints seem to be uniquely solved by $\varphi_0$ and $\psi_0$ the vector representation of $\so(16)$ with components $\psi_0^I$ and $\varphi_0^I$ for $I=1,\dots,16$.
The action of the generators $A^{IJ}_{0,0}$ is given by $\so(16)$-covariance, while the action of the generators $A^{IJ}_{0,1}$ is
\begin{align}
\label{eq:action_on_16_16}
A^{IJ}_{0,1}\, \varphi_0^I \eq A^{IJ}_{0,0} \,\varphi_0^I && A^{IJ}_{0,1} \psi_0^I \eq -A^{IJ}_{0,0} \psi_0^I\, . 
\end{align}
In fact, the tensor product of these two representation has the right dimension $16\cdot 16 = 256$ and decomposes into representations which also appear in \eqref{eq:Basic_e9_Decomposed}
\begin{align}
\sorep{16}{16}\otimes \sorep{16}{16} \eq \sorep{1}{16}\oplus\sorep{120}{16}\oplus \sorep{135}{16}\, .
\end{align}
The $\sorep{120}{16}$ is the adjoint representation and the $\sorep{135}{16}$ representation appears in the symmetric product of two adjoint representations $\sorep{120}{16}\otimes_{\textrm{sym}} \sorep{120}{16} \eq \sorep{1}{16} \oplus \sorep{135}{16}\oplus \sorep{1820}{16}\oplus \sorep{5403}{16}$.
The singlet is
\begin{align}
\label{eq:En_Reps_0_1}
\psi_0^K\otimes \varphi_0^K \otimes 1 \quad \rightarrow\quad  \sorep{1}{16}\, 
\end{align}
and the double parabolic generators \eqref{eq:action_on_16_16} map this singlet into the $\sorep{120}{16}$ representation and the $\sorep{135}{16}$ representation
\begin{alignat}{3}
\label{eq:En_Reps_0_2}
A^{IJ}_{0,1}\left( \psi_0^K\otimes \varphi_0^K\right) 
\aeq  
4 \, \psi_0^{[J}\otimes \varphi_0^{I]} \quad &&\rightarrow\quad \sorep{120}{16}\nn\\
A^{KL}_{0,1}\left( \psi^{[J}\otimes \varphi^{I]} \right) 
\aeq  
4\,\delta^{LI} \psi_0^{(K}\otimes \varphi_0^{J)} 
\quad 
&&\rightarrow\quad \sorep{1}{16}\oplus \sorep{135}{16}\,.
\end{alignat}
as it is predicted in the general analysis in \eqref{eq:GeneralPar0}. This is the $\so(16)$-decomposition of $\verQ(0)$.

\paragraph{Parabolic level one.}
Next, for $N\geq 1$ we analyze $\verQ(N)$ at parabolic level one with elements 
\begin{align}
\left\lbrace \be_\gamma, \, \Qp^i_{1,0}\be_\gamma,\, \Qp^i_{1,1}\be_\gamma \in\verQ(1) \forwhich \gamma \in \Qr,\, 1\leq i\leq 8 \right\rbrace
\end{align}
There are $256 \, + \, 2 \cdot 8 \cdot 256$ inequivalent elements. However, $256$ of these elements are of parabolic level zero and therefore,  $2 \cdot 8 \cdot 256$ of these elements are of of parabolic level one. This implies restrictions on the possible $\so(16)$-representations at level one in $\Phi$ and $\Psi$. 

By the construction in \cite{Kleinschmidt:2021agj}, the level one representations of $\Phi$ and $\Psi$ is the Verma module action of $S^A_{1,0}$ on $\varphi^I_0$ and $\psi^I_0$
\begin{align}
\label{eq:action_S1}
S^A_{1,0}\, \varphi^I_0 \aeq \Gamma^I_{A,\dot A} \varphi_1^{\dot A} \, + \, \varphi_1^{AI} &&
S^A_{1,0}\, \psi^I_0 \eq \Gamma^I_{A,\dot A} \psi_1^{\dot A} \, + \, \psi_1^{AI}\, ,
\end{align}
where $\varphi_1^{\dot A}$ and $\psi^{\dot A}_1$ are the conjugate spinor representations $\sorep{\overline{128}}{16}$, $\varphi_1^{IA}$ and $\psi^{IA}_1$ are the $\sorep{\overline {1920}}{16}$ representations and $\Gamma^I_{A\dot A}$ are the $\so(16)$ gamma matrices.

Then, at level one, the tensor product of $\Phi$ and $\Psi$ has dimensions $2 \cdot 16 \cdot (128 + 1920)>2 \cdot 8\cdot 256$ which is clearly larger than the dimension of $\verQ(N)$ at level one. However, the $\sorep{1920}{16}$ representations $\varphi_1^{IA}$ and $\psi^{IA}_1$ can be quotiented from $\Phi$ and $\Psi$ and therefore they can be set to $0$ \cite{Kleinschmidt:2021agj}. Hence, the representations $\Phi$ and $\Psi$ are up to parabolic level one
\begin{align}
\Phi^{I,\dot{A},\dots} \eq \left(\varphi_0^I,\, \varphi_1^{\dot A},\, \ldots\right)&& \Psi^{I,\dot{A},\dots} \eq \left(\psi_0^I,\, \psi_1^{\dot A},\, \ldots\right)\,
\end{align}  
and the action of the double parabolic algebra is uniquely given by $\so(16)$-covariance, \eqref{eq:action_on_16_16} and \eqref{eq:action_S1}. On the tensor product, up to parabolic level one, the double parabolic algebra acts by
\begin{align}
\label{eq:En_Reps_1_2}
S^{A}_{1,0}\left( \psi_0^I\otimes \varphi_0^J\right) 
\aeq  
\Gamma_{A,\dot A}^J\,\psi_0^{I}\otimes  \varphi_1^{\dot A} \,+\, \Gamma_{A,\dot A}^I\,\psi_1^{\dot A}\otimes  \varphi_0^{J}
\nn\\
S^{A}_{1,1}\left( \psi_0^I\otimes \varphi_0^J\right) 
\aeq  
\Gamma_{A,\dot A}^J\,\psi_0^{I}\otimes  \varphi_1^{\dot A} \,-\, \Gamma_{A,\dot A}^I\,\psi_1^{\dot A}\otimes  \varphi_0^{J}.
\end{align}
Indeed, these representations on the right hand side are exactly the representations to expect in $\verQ(N)$ at parabolic level one because
\begin{align}
2\cdot\left( \sorep{16}{16}\otimes \sorep{\overline {128}}{16}\right) \aeq 2\cdot\left(\sorep{128}{16}\oplus \sorep{1920}{16}\right)\nn\\
 \,\simeq& \;
 2\cdot 2048 \,\simeq\, \dim(\verQ(1))-\dim(\verQ(0))\, .
\end{align} 
This fixes $\Phi$ and $\Psi$ up to parabolic level one and together with parabolic level zero determines the $\so(16)$-decomposition of $\verQ(1)$.

\paragraph{Parabolic level two.} The elements in $\verQ(N)$ for $N\geq 2$ which have contribution to parabolic level $2$ are
\begin{align}
\left\lbrace \be_\gamma, \, \Qp^i_{1,0}\be_\gamma,\, \Qp^i_{1,1}\be_\gamma, 
\, \Qp^i_{1,0}\Qp^j_{1,0}\be_\gamma,\, \Qp^i_{1,1}\Qp^j_{1,0}\be_\gamma, \, \Qp^i_{1,1}\Qp^j_{1,1}\be_\gamma,\, \Qp^i_{1,1}\be_\gamma\forwhich \gamma \in \Qr,\, 1\leq i,j\leq 8 \right\rbrace
\end{align}
These are $256 + 2 \cdot 8 \cdot 256 + (2\cdot 36 + 64)\cdot 256$ independent elements, however there are $256 + 2\cdot 8\cdot 256$ elements at parabolic level zero and one, such that 
\begin{align}
(2\cdot 36 + 64)\cdot 256 \eq 34816 
\end{align}  
independent elements remain at parabolic level $2$. Let us build the representations $\Phi$ and $\Psi$ such that the tensor product has at level $2$ the same number of elements. 

In \cite{Kleinschmidt:2021agj} Section 5.3, the second parabolic level of the $\Phi$ and $\Psi$ is evaluated and it consists of the $\so(16)$-vector representation and the three form $\sorep{560}{16}$, such that $\Phi$ and $\Psi$ are 
\begin{align}
\label{eq:Phi_And_Psi}
\Phi^{I,\dot{A},J,[KLM]} \eq \left(\varphi^I_0,\varphi^{\dot{A}}_1,\varphi^{J}_2,\varphi^{[KLM]}_2\right)\, , &&
\Psi^{I,\dot{A},J,[KLM]} \eq \left(\psi^I_0,\psi^{\dot{A}}_1,\psi^{J}_2,\psi^{[KLM]}_2\right)\,.
\end{align} 
The tensor product of these two representations at parabolic level $2$ has the components
\begin{align}
\varphi^I_0 \otimes \psi^{J}_2\, , \quad \varphi^I_0 \otimes \psi^{[KLM]}_2\, , \quad \varphi^{\dot{A}}_1\otimes\psi^{\dot{A}}_1\, , \quad\varphi^{J}_2\otimes\psi^I_0\, , \quad\varphi^{[KLM]}_2\otimes\psi^I_0
\end{align}
which are indeed the number of elements of $\verQ(N)$ at parabolic level $2$
\begin{align}
16\cdot 16 \, + \, 16\cdot 560 \, + \, 128\cdot 128 \, + \, 16\cdot 16 \, + \, 560\cdot 16 \eq 34816\,.
\end{align}
Together with the decomposition of parabolic level zero and one above this is the $\so(16)$-covariant form of $\verQ(2)$.

This is a quite remarkable result and shows how the basic representation projects onto a tensor product of two different `spinor'-representations of the single parabolic algebra. 
However, this has even more consequences to obtain quotients of subrepresentations of $\verQ(N)$ which are also projections of the basic representation. 
Usually, if these quotients by subrepresentations do not correspond to a representation $\verQ(M)$ with $M<N$, they are very hard to identify since they must the subrepresentations must be subrepresentations of the full double parabolic algebra and not only of the single parabolic algebra. However, since we identified $\verQ(N)$ for $N\leq 2$ with the tensor product of $\Phi$ and $\Psi$, the tensor product of subrepresentations of $\Phi$ and $\Psi$ are also subrepresentations of $\verQ(N)$. 
The advantage is that the subrepresentations of $\Phi$ and $\Psi$ are much simpler to identify, because these are representations of the single parabolic algebra. Let us discuss possible subrepresentations of interest for applications to supergravity in two dimensions.

\subsubsection{Application to Supergravity in two dimensions}
\label{subsubsec:Applications_to_Supergravity}
Different subrepresentations of $\verQ(2)$ and $\verQ(4)$ find applications in maximal supergravity in two dimensions. One application is the supersymmetry of the gauge fields of the theory while another application is the decomposition of the embedding tensor in fermion bilinears. We first introduce the $\enn$ and $\kenn$ representations for the fields of interest. 

For each of the two space time indices, the gauge fields of supergravity in two dimensions transform in the basic representation \cite{Samtleben:2007an}. The fermions of the supergravity theory are $\Phi\otimes \mathcal C_2$ with $\Phi$ in \eqref{eq:Phi_And_Psi} but without the three form at parabolic level $2$ and $\mathcal C_2$ accounts for the two chiralities of spinors in two dimensions \cite{Nicolai:2004nv}. The supersymmetry parameter of the theory is $\epsilon^I\simeq \psi_0^I \otimes \mathcal{C}_2$, which is $\Psi \otimes \mathcal{C}_2$ in \eqref{eq:Phi_And_Psi} but keeping only the $\so(16)$-vector representation at parabolic level zero.
Then, schematically the supersymmetry of the vector fields is associated to the homomorphism $\Gmap_2$ in \eqref{eq:GMap} which maps onto the tensor product of 
\begin{align}
\epsilon^K \otimes \Phi^{I,\dot A, J} \otimes \mathcal C_2\, . 
\end{align} 
Only one factor of $\mathcal C_2$ remains because the two factors of $\mathcal{C}_2$ are multiplicative.

Another interesting object necessary for gauged maximal supergravity in two dimensions is the embedding tensor. On the one hand, the embedding tensor takes values in the basic representation of $\enn$ \cite{Samtleben:2007an} on the other hand it decomposes into fermion bilinears. This decomposition is the homomorphism $\Gmap_4$ in \eqref{eq:GMap} mapping on the tensor product of the fermions 
\begin{align}
\Psi^{K,\dot B, L}\otimes \Phi^{I,\dot A, J}\otimes \mathcal{C}_2
\end{align}
but with the same $\so(16)$-components $\psi_k^X = \varphi_k^X$. This is a subrepresentation of $\verQ(4)$. 

In a subsequent paper, we extend this schematic realization of supersymmetry and the decomposition of the embedding tensor to a full description.

\begin{appendices}

\section{Representations of double parabolic algebra}
\label{app:Reps}
The representations of the single parabolic algebra are constructed in \cite{Kleinschmidt:2021agj} from a Verma module starting as a $\kfinite$-representation at parabolic level zero. By the parabolic grading, every $\kfinite$-representation in this Verma module generates a subrepresentation. These subrepresentations can be quotiented from the module to obtain finite dimensional representations. 
This procedure extends to the double parabolic algebra
\begin{align}
\left(\kfinite \otimes \reals[u^{2n}]
\, \oplus\,  \pfinite \otimes u\reals[u^{2n}]\right)\otimes \mathcal C_2
\end{align}
but in both parameters $u$ and $1\neq r\in\mathcal C_2$. However, $r$ has only a $\ints_2$ grading and requires additional considerations. 
The Lie bracket is multiplicative in $u$ and $r\in \mathcal{C}_2$ and has therefore a graded structure in $u$ and a $\ints_2$ graded structure in $r$. However, the universal enveloping algebra is $not$ multiplicative in $u$ and $r$, which equips the universal enveloping algebra with a natural graded structure in $u$ and in $r$. Therefore, representations of the double parabolic algebra are graded with respect to $u$ and $r$. To construct these representations we extend the construction in \cite{Kleinschmidt:2021agj} to the double parabolic algebra. 

Therefore, let us use $\Nfrak$ for the single parabolic algebra, which is obtained from $\kfrakpar$ by setting $r=1$ and the
universal enveloping algebra of the single parabolic generators with positive power in $u$ is $\Uc(\Nfrak_+)$ in \cite{Kleinschmidt:2021agj}. This extends for the double parabolic algebra to the universal enveloping algebra of double parabolic algebra with positive power in $u$ or in $r$
\begin{align}
\Uc({\kfrakpar}_+) = \Uc\left( \Nfrak_+\cdot r\,\oplus\, \Nfrak_+\,\oplus\, \textrm{Sym}\left(\,\kfinite\cdot r\,\right)\right)\,. 
\end{align}
A basis of $\Uc({\kfrakpar}_+)$ is given by products of elements of $ \Nfrak_+\cdot r$ to the left of $\Nfrak_+$, which are to the right of symmetric products in $ \kfinite\cdot r$. 

Let $\mathfrak{B}_{\Nfrak_+}$ be a PBW basis of $\Uc(\Nfrak_+)$ and let $\tilde{\mathfrak{B}}_{\Nfrak_+}$ 
be the set $\mathfrak{B}_{\Nfrak_+}$ but with each generator of $\Nfrak_+$ multiplied by $r$. Let $\tilde{\mathfrak{B}}_{\kfinite}$ be the PBW 
basis of $\kfinite$ but with each generator multiplied by $r$, then a PBW basis of $\Uc({\kfrakpar}_+)$ is given by the ordered products of 
\begin{align}
\mathfrak{B}_{{\kfrakpar}_+} \eq \tilde{\mathfrak{B}}_{\Nfrak_+}\times \mathfrak{B}_{\Nfrak_+}\times \tilde{\mathfrak{B}}_{\kfinite}\,.
\end{align}

By the construction in \cite{Kleinschmidt:2021agj}, the representations of the double parabolic algebra $\kfrakpar$ are constructed from the action of $\Uc({\kfrakpar}_+)$ on a $\kfinite$-representation at level $u^0$ and $r^0$.

The parabolic grading in $r$ and $u$ allows to identify infinite many ideals by restricting to a maximal level in $u$ and $r$ of a representation. These ideals allow for finite dimensional representations of $\kfrakpar$, by setting the action of the ideals to zero. For example, restricting to $r^0$ reduces the representations of the double parabolic algebra to representations of the parabolic algebra $\Nfrak$. Finite dimensional representation of the parabolic algebra are obtained for example by considering a maximal level $u^N$.

\section{Generalized summation}
\label{app:PolySum}
It is useful to have a notation of generalized sums over polynomials $p(k)$, where the upper limit can be lower than the lower limit, but the sum remains finite. For a meaningful definition of generalized sum Faulhaber's formula is useful.

Let $n,N\in\nats$ and let $p\in\reals[k]_N$ be a polynomial of degree $N$, then the sum
\begin{align}
P(n) \eq \sum_{k=1}^n p(k) \in n\,\reals[n]_{N}
\end{align}
is a polynomial of degree $N+1$ on the domain $n\in\nats$ with a zero at $n=0$ \cite{Jacobi1834}.\footnote{The full formula for sums over monomials is called the Faulhaber formula 
\begin{align}
\sum_{k=1}^{n} k^p 
\eq \frac{1}{p+1}\sum_{r=0}^p
\binom{p+1}{r}B_r n^{p+1-r} 
\end{align}
where $B_r$ are the Bernoulli numbers.
It was proven by L. Tits in 1923.}
The polynomial $P$ extends uniquely to the domain $n\in\ints$ as a polynomial $\bar P\in n\reals[n]_N$ of degree $N+1$. For this polynomial we derive an identity, therefore, let $n_0\in\nats$, $1\leq n\leq n_0$, then in
\begin{align}
P(n) \eq P(n_0) -\sum_{k=n+1}^{n_0} p(k),
\end{align}
the right hand side is a polynomial $P'\in\reals[n]_{N+1}$ on the domain $n\in\ints, \; n\leq n_0$. The unique polynomial extension of $P'$ from the domain $n\in\ints,\; n\leq n_0$ to $n\in\ints$ defines $\bar P'$, but because $n_0$ is arbitrary, $\bar P'(n) \eq \bar P(n)$ for all $n\in\ints$ and therefore $\bar P' = \bar P$. For all $n\in\nats_0$ this implies the identity
\begin{align}
\label{eq:Summation_Negative_Limit}
\bar P(-n) \eq -\sum_{k=0}^{n-1}p(-k)\,.
\end{align}
Now, we define the generalized sum over the polynomial $p$ of degree $N$
\begin{align}
\label{eq:Generalized_Sum_Def}
\gsum{m=1}{n} p(m) \eq \bar P(n)\in n\,\reals[n]_{N}
\end{align}
as the polynomial extension of $\sum_{m=1}^n p(m)$ from the naturals to the integers $n\in\ints$. 
By the above argument, the generalized sum satisfies for every $n\in\nats$
\begin{subequations}
\begin{align}
 \gsum{k=1}{-n} p(k) \aeq - \sum_{k=0}^{n-1} p(-k)\label{eq:Generalized_Sum_Id}\\
 \gsum{k=1}{0} p(k) \aeq - \gsum{k=0}{-1} p(-k)=0\label{eq:Generalized_Sum_Id_0}
\end{align}
\end{subequations}

\section{Definitions and Notations}
\label{app:Cosocle_Filtration}
We introduce some notations and frequently used definitions in this work.

\paragraph{Notation sets and families.}
We write the set of elements with an additional index by
\begin{align}
\label{eq:SetNotation}
\{x\}_I \aeq \{x_n\forwhich n\in I\}\nn\\
\{x\in X\}_I \aeq \{x_n|x\in X,\,n\in I\}\, ,
\end{align}
but if $I$ is clear from the context, then we also write $\{x\} \eq \{x\}_I$.

\paragraph{Embedding and projection.} For two vector spaces $V,W$ and a homomorphism $\Gmap':\,V\to W$ we say that $V$ is \textit{embedded }in $W$ if $\Gmap'$ is injective and if $\Gmap'$ is surjective we say that $W \simeq V/(\mathrm{Ker}(\Gmap')$ is a \textit{projection} of $V$. If the vector spaces are representations of a Lie algebra $\mathfrak{l}$ and $\Gmap'$ respects the action of $\mathfrak{l}$ we also call the embedding (resp. projection) a $\mathfrak{l}$-embedding (resp. $\mathfrak{l}$-projection). If it is clear from the context we also avoid the symbol $\mathfrak{l}$ and simply write embedding (resp. projection). 

\paragraph{Lie algebra structure and fineness.}
For a Lie algebra $\mathfrak{l}$ with a representation $V$, a \textit{$\mathfrak{l}$-structure} of $V$ is a family of $\mathfrak{l}$-representations $V_i\subset V$. For two $\mathfrak{l}$-structures $A$ and $B$, we call $A$ \textit{finer} than $B$ if $B\subset A$. If for a $\mathfrak{l}$-structure $A$ holds $B\subset A$ for all $\mathfrak{l}$-structures $B$, we call $A$ the \textit{finest $\mathfrak{l}$-structure.}

\paragraph{Cosocle Filtration and composition series.}
Let us introduce the cosocle filtration and the infinite composition series in a series of definition and implications. For further insides see also \cite{Hu_2022}.
Therefore, let $\mathfrak{l}$ be a Lie algebra with a representation $V$. 
\begin{enumerate}[label=\alph*.)]
\item  A subrepresentation $V\neq W\subset V$ is \textit{maximal} if for any other subrepresentation $U\in V$ hold that $W \subset U$ implies $U=W$ or $U=V$.\\
$\Leftrightarrow$ If $W\subset V$ is maximal then, $\frac{V}{W}$ is simple.
\item  The \textit{radical} $\mathrm{rad}(V)$ of $V$ is the intersection of all maximal subrepresentations of $V$. 
$\Leftrightarrow$ The quotient $\frac{V}{\mathrm{rad}(V)}$ is maximal semisimple.    
\item  The \textit{cosocle} of $V$ is the quotient by the radical $\mathrm{cosoc}(V) = \frac{V}{\mathrm{rad}(V)}$. 
$\Leftrightarrow$ The cosocle is the maximal semisimple quotient of $V$.
\item  The \textit{cosocle filtration} resp. (\textit{radical filtration}) of $V$ is inductively defined. $V_{-1} = V$ and for $i\in\nats_0$, $V_i$ is the kernel of the projection of $V_{i-1}$ to its cosocle. 
$\Leftrightarrow$ $V_0$ is the radical of $V$ and for $i\in\nats_0$, $V_{i+1}$ is the radical of $V_{i}$. The filtration has \textit{finite length} if $J\in\nats$ exists, such that $V_{j}=0$ for $j>J$. Then $V_J$ is semisimple.
\item \label{itm:Composition_Series} A \textit{composition series} of $V$ is a family of subrepresentations $(V_i)_{0\leq i\leq n}$ for an $n\in\nats$ such that
\begin{align}
V=V_0\supset V_{1} \supset\ldots \supset V_n=\{0\}\,
\end{align}
and $V_i/V_{i+1}$ is semisimple. If $V_{n-1}\neq \{0\}$, then $n$ is the \textit{length} of the composition series.
\item  An \textit{infinite composition series} of $V$ are subrepresentations $(V_i)_{i\in \nats}$ such that
\begin{enumerate}[label=\roman*.)]
\item	$V_{0} = V \,, \qquad V_i\supset V_{i+1}\, , \qquad V_{i}\neq\{0\}$. 
\item The inverse limit of $(V_j)_{j\in\nats}$ with the natural embedding $ V_{i+1}\subset V_{i}$ is properly defined and $\lim\limits_{\substack{\longleftarrow \\ j\to \infty}} V_j \eq 0$
\item The quotients $V_i/V_{i+1}$ are semisimple.
\end{enumerate}
This is the proper generalization of \ref{itm:Composition_Series}.
The cosocle filtration is an infinite composition series if \textit{i.} and \textit{ii.} hold.
\end{enumerate}

\section{Expansions in first orders}
\label{app:ExpansionFirst order}
We break down commonly used expressions into explicit expansions to make further calculations easier and to gain some intuition about the underlying concepts.

\paragraph{Lie algebra homomorphism.}
The Lie algebra homomorphism from the maximal compact subalgebra to the double parabolic algebra is explicitly on the first $3$ loop generators
\begin{align}
\label{eq:LieHomDouble_Expl}
\rho(\Ec^\alpha_{0}) \aeq P_{0,0}^\alpha\nn\\
\rho(\Ec^\alpha_{\pm 1}) \aeq P_{0,1}^\alpha+2\sum_{k\geq 1} (\mp)^k P_{k,1}^\alpha\nn\\
\rho(\Ec^\alpha_{\pm 2}) \aeq P_{0,0}^\alpha+4\sum_{k\geq 1} (\mp)^k k P_{k,1}^\alpha\nn\\
\rho(\Ec^\alpha_{\pm 3}) \aeq P_{0,1}^\alpha+2\sum_{k\geq 1} (\mp)^k(1+2k^2) P_{k,1}^\alpha\\
&\ldots\nn
\end{align}
\begin{align}
\rho(\Hc^\alpha_{\pm 1}) \aeq -2\sum_{k\geq 1} (\pm)^k \Qp_{2k+1,1}^\alpha\nn\\
\rho(\Hc^\alpha_{\pm 2}) \aeq -4\sum_{k\geq 1} (\pm)^k (2k+1) \Qp_{2k+1,1}^\alpha\nn\\
\rho(\Hc^\alpha_{\pm 3}) \aeq -2\sum_{k\geq 1} (\pm)^k(8k^2+8k+3) \Qp_{k,1}^\alpha\\
&\ldots\nn
\end{align}

\paragraph{Schur Polynomials.}
Up to index $4$, the Schur polynomials are
\begin{align}
\label{eq:Schur_x_Examples}
S_{n\leq-1}(\{x\}) \aeq 0
\nn\\
S_0(\{x\}) \aeq 1
\nn\\
S_{1}(\{x\})\aeq x_{-1}\nn\\
S_2(\{x\}) \aeq  \half x_{-2} + \half x_{-1}^2\nn\\
S_3(\{x\}) \aeq  \tfrac{1}{3} x_{-3} + \half x_{-1}x_{-2} +\tfrac{1}{6}x_{-1}^3 \nn\\ 
S_4(\{x\}) \aeq  \tfrac{1}{4} x_{-4} + \tfrac{1}{3} x_{-1}x_{-3} +\tfrac{1}{8}x_{-2}^2+ \tfrac{1}{4} x_{-1}^2x_{-2}+ \tfrac{1}{24} x_{-1}^4 \\
&\ldots\nn
\end{align}

\paragraph{\textit{H}-Schur Polynomials.}
The $\Hc$-Schur polynomials generate the Schur polynomials by action on $1$. They are given by
\begin{align}
\label{eq:Schur_H_Examples}
S_{0}^{\Hc}(\{\Hc^\alpha\}) \aeq  1\, , \nn \\
S_{1}^{\Hc}(\{\Hc^\alpha\}) \aeq  -2\Hc^\alpha_1\, , \nn \\
S_{2}^{\Hc}(\{\Hc^\alpha\}) \aeq  1+2(\Hc^\alpha_1)^2 - \Hc^\alpha_2\, , \nn \\
S_{3}^{\Hc}(\{\Hc^\alpha\}) \aeq  -2\Hc^\alpha_1-\tfrac{4}{3}(\Hc^\alpha_1)^3 +2 \Hc^\alpha_1\Hc^\alpha_2- \tfrac{2}{3} \Hc^\alpha_3\, , \\
\ldots\nn
\end{align}

\paragraph{Extended $\Qp$-Schur Polynomials.}
The extended $\Qp$-Schur polynomials $\bar S^\Qp_{\alpha,a}(n)$ are satisfy for $n\geq 0$ the identity $\bar S^\Qp_{\alpha,0}(2n) = \rho(S^{\Hc}_{2n}(\{\Hc^\alpha\})$ and $\bar S^\Qp_{\alpha,1}(2n+1) = \rho(S^{\Hc}_{2n+1}(\{\Hc^\alpha\}))$, they are given up to parabolic level $2$ by
\begin{align}
\label{eq:App_Q_Schur_Poly}
\bar S_{\alpha,0}^{\Qp}(n) \aeq 1 \,+\, 2n\Qp^\alpha_{1,0} \,+\, (2n+n^2)\Qp^\alpha_{1,1}\Qp^\alpha_{1,1} \,+\, (-2n+n^2) \Qp^\alpha_{1,0}\Qp^\alpha_{1,0}\,+\, \ldots\nn\\[2pt]
\bar S_{\alpha,1}^{\Qp}(n) \aeq  2(n+1)\Qp^\alpha_{1,1} \,+\, (-2+2n^2)\Qp^\alpha_{1,1}\Qp^\alpha_{1,0} \,+\, \ldots
\end{align}

\paragraph{Decomposition of basic representation $\basic$ of $\enn$.}
The basic representation of $\enn$-decomposes into $\eee$-representations at the different levels. For the first five levels this decomposition is given by
\begin{align}
\label{app:Decompose_enn_Basic}
\basic &\eq  \left(\mathbf{1}_{\eee}\right)_0\; \oplus\; \left( \mathbf{248}_{\eee}\right)_1\; \oplus\;\left(\mathbf{1}_{\eee}\;\oplus\; \mathbf{248}_{\eee} \oplus \;\mathbf{3875}_{\eee}\right)_2
\nn\\[5pt]
&\oplus\; \left(\mathbf{1}_{\eee}\oplus 2\cdot \mathbf{248}_{\eee}\oplus \mathbf{3875}_{\eee}\oplus \mathbf{30380}_{\eee}\right)_3 
\nn\\
&\oplus\; \left(
 2\cdot \mathbf{1}_{\eee}\oplus 3\cdot \mathbf{248}_{\eee}\oplus 2\cdot \mathbf{3875}_{\eee}\oplus \mathbf{30380}_{\eee}\oplus \mathbf{27000}_{\eee}\oplus \mathbf{147250}_{\eee} \right)_4
\nn\\
&\oplus\; \;\ldots 
\end{align}
The individual $\eee$-representation can be further decomposed into the maximal compact subalgebra $\so(16)$ of $\eee$
\begin{align}
\mathbf{1}_{\eee} \quad &\rightarrow\quad \mathbf{1}_{\so(16)}\nn\\
\mathbf{248}_{\eee} \quad &\rightarrow\quad \mathbf{120}_{\so(16)}\oplus\mathbf{128}_{\so(16)}\nn\\
\mathbf{3875}_{\eee} \quad &\rightarrow\quad \mathbf{135}_{\so(16)}\oplus\mathbf{1820}_{\so(16)}\oplus\mathbf{1920}_{\so(16)}\nn\\
\mathbf{30380}_{\eee} \quad &\rightarrow\quad \mathbf{120}_{\so(16)}\oplus\mathbf{1920}_{\so(16)}\oplus\mathbf{7020}_{\so(16)}\oplus\mathbf{8008}_{\so(16)}\oplus\mathbf{13312}_{\so(16)}\nn\\
\ldots
\end{align}

\section{Details in proofs}
\label{app:Appendix_Details}
For the convenience of the interested reader, this appendix provides more details and explicit calculations of various proofs, which are not given in full detail in the main text.
\paragraph{Commutation of $U\in\Uc(\cheisenberg)$ with $\Ec_n^r$.}
The commutator of $\Ec^i_n$ with $\prod_{l=1}^N \Hc^{j_l}_{n_l}$ is the group like action of $\Ec^i_n$ on the universal enveloping of the compact Heisenberg algebra. Therefore, the commutator of $\Ec^i_n$ with a product is the product of the commutators. For $A_{ij}=0$, $\Ec^i_n$ and $\Hc^j_m$ commute. It remains to repeatedly evaluate the commutator of $\Ec_N^i$ and products of $\Hc_m^j$ for $A_{ij} = -1,2$. Here, we evaluate the commutator for $A_{ij} = 2$. Therefore, the set
\begin{align}
\Pc_N \acoloneqq \Set{p = (p_1,\ldots ,p_N) \in \ints_2^N \forwhich p_l\in \{0,\, 1\}}\, ,
&|p| \ecoloneqq \sum_{l=1}^N p_l\; .
\end{align} 
is useful. In this notation, the commutator is
\begin{align}
\label{eq:ComutatorEcUHc}
\left[\Ec^i_n\, , \; \prod_{i=1}^N \Hc^i_{n_l}\right] \aeq (-1)^N \left(\sum_{p \in P_N}(-1)^{|p|}(\half  \Ec^i_{n+\sum_{i=1}^N(-1)^{p_l}n_l}\right)\nn\\
\, &+ \,  (-1)^{N-1}\sum_{k_1=1}^N\Hc^i_{n_{k_1}} \left(\sum_{p \in P_{N-1}} (-1)^{|p|} \Ec^i_{n+\sum_{k_1\neq i=1}^n(-1)^{p_l}n_l}\right)\nn\\
\, &+ \,  (-1)^{N-2}\sum_{k_1=1}^N\sum_{k_2=k_1+1}^N\Hc^i_{n_{k_1}}\Hc^i_{n_{k_2}} \left(\sum_{p \in P_{N-1}}(-1)^{|p|}  \Ec^i_{n+\sum_{k_{1,2}\neq i=1}^n(-1)^{p_l}n_l}\right)\nn\\
\, & + \, \ldots \nn\\
\, &+ \,  (-1)^1\sum_{k_1=1}^N\sum_{k_2=k_1+1}^N\ldots \sum_{k_{N-1}=k_{N-2}+1}^N\Hc^i_{n_{k_1}}\Hc^i_{n_{k_2}}\ldots \Hc_{n_{k_{N-1}}}\nn\\ 
& \hspace*{5.3cm} \left(\sum_{p \in P_{1}} (-1)^{|p|}\Ec^i_{n+\sum_{k_{1,2,\dots N-1}\neq i=1}^n(-1)^{p_l}n_l}\right)\, ,
\end{align}
where $k_{1,2,\dots} \neq i$ means $i \neq k_1$, $i\neq k_2$, ... .
The commutator for $A_{ij}=-1$ has additional factors of $-\half$.

\paragraph{Proof Lemma \ref{lem:SchurH}.}
The evaluation of \eqref{eq:ProofLemSH} in the proof of Lemma \ref{lem:SchurH} is in detail
\begin{align}
\Sh_N&(\HSet^i) \eq \frac{2}{N+1}\sum_{n=1}^{N+1} \left(\Sh_{N-2n+1}-\Hc_nS_{N-n+1} \right)\nn\\
&\eq\sum_{p\geq 0}\sum_{k\in\Kc_{p}}\delta_{(p+N+1)mod2,0}\sum_{n=1}^{N+1} \frac{2}{N+1}\Big(\sum_{n=1}^{N+1}\Theta(N-2n+1-deg(k)) 
 -(\frac{nk_n}{2})\Big)d_k(\Hc^i)^k
\nn\\
&\eq
\sum_{p\geq 0}\sum_{k\in\Kc_{p}}\delta_{(p+N+1)mod2,0}\Theta(N+1-p) \frac{2}{N+1}\Big(\left\lfloor\frac{N+1-p}{2}\right\rfloor 
 +\sum_{n=1}^{N+1}(\frac{nk_n}{2})\Big)d_k(\Hc^i)^k
\nn\\
&\eq
\sum_{p\geq 0}\sum_{k\in\Kc_{p}}\delta_{(p+N+1)mod2,0}\Theta(N+1-p) \frac{2}{N+1}\Big(\frac{N+1-p}{2}
 +\frac{p}{2}\Big)d_k(\Hc^i)^k
\nn\\
&\eq
\sum_{p\geq 0}\sum_{k\in\Kc_{p}}\delta_{(p+N+1)mod2,0}\Theta(N+1-p) d_k(\Hc^i)^k\, .
\end{align}

\paragraph{Proof Proposition \ref{prop:VerQRepresentation}.}
To prove the commutation of $\Ec_m^\alpha$ with $\Ec_n^\beta$ on $\verQ(N)$, we argued with the commutation relations of the basic representation this translates to the commutation relations on $\verQ(N)$. For the readers convenience, we illustrate this in more detail here. 

Let us define a map $\tilde \Gmap_N$ from the basic representation onto the $\verQ(N)$ by
\begin{align}
\tilde \Gmap_N\left(e^{\gamma}\otimes f(\{\alpha\})\right)
\eq 
\Fc_f^{N,\Qp}(\QSet) \be_{\gamma}
\end{align}
for $\gamma \in \Qr$ and $f$ a polynomial.
By the definition of the generating elements $\Fc_f^\Qp$ the homomorphism $\tilde \Gmap_N$ commutes with the compact Heisenberg algebra. By the action of $\Ec_m^\alpha$ on $\be_{\gamma}$  \eqref{eq:Action_Ec_on_be_alpha} and by its action on the maximal states \eqref{eq:Ec_nOnMaximalState}, $\tilde \Gmap_N$ commutes with $\Ec_n^\alpha$ on maximal states. Additionally, by construction $\Ec_m^\alpha$ and the parabolic Heisenberg algebra satisfy the commutation relations. This is enough to prove the commutation relations of $\Ec_m^\alpha$ and $\Ec_n^\beta$ on $\be_\gamma$
\begin{align}
\label{eq:app_prop_addition_1}
\Ec_m^\alpha\Ec_n^\beta \, \be_\gamma \eq  \epsilon(\beta,\gamma)\left(\Ec_m^\alpha S^{\Hc}_{-n-(\beta|\gamma)-1}\tilde \Gmap_N(e^{\gamma+\beta}\otimes 1)+\Ec_m^\alpha S^{\Hc}_{n+(\beta|\gamma)-1}\tilde \Gmap_N(e^{\gamma-\beta}\otimes 1)\right)
\end{align}
Here, let us just consider the first term in the sum, the other term works the same. We use the commutation relations of $\Ec_m^\beta$ and the compact Heisenberg algebra in \eqref{eq:Commutation_Ecn_Heisenberg} with $J' = J(S^{\Hc}_{-n-(\beta|\gamma)-1})$, $U' = U(S^{\Hc}_{-n-(\beta|\gamma)-1})$.
Then, the first term becomes
\begin{align}
\sum_{j=1}^{J'}{U'}^\alpha_j \Ec_m^\alpha \tilde \Gmap_N\left(e^{\gamma+\beta}\otimes 1\right) 
\aeq 
\tilde \Gmap_N\left(\sum_{j=1}^{J'}{U'}^\alpha_j \Ec_m^\alpha e^{\gamma+\beta}\otimes 1\right)
\eq \tilde \Gmap_N\left(\Ec_m^\alpha \Ec_n^\beta e^{\gamma}\otimes 1\right)\,.
\end{align}
The same argument holds for the second term in \eqref{eq:app_prop_addition_1} as well. Therefore, for all $\gamma\in \Qr$
\begin{align}
\left(\Ec_n^\alpha\Ec_m^\beta - \Ec_m^\beta\Ec_n^\alpha - \left[\Ec_n^\alpha\, ,\, \Ec_m^\beta\right]\right)\be_\gamma
\eq 
\tilde \Gmap_N\left(\left(\Ec_n^\alpha\Ec_m^\beta - \Ec_m^\beta\Ec_n^\alpha - \left[\Ec_n^\alpha\, ,\, \Ec_m^\beta\right]\right)e^\gamma\otimes 1\right) \eq 0\,.
\end{align}
This proves the third step of the proof of Proposition \ref{prop:VerQRepresentation} in detail.

\end{appendices}

\printbibliography

\end{document}